\documentclass{amsart}
\numberwithin{equation}{section}

\usepackage{bbm}
\usepackage{comment}
\usepackage{verbatim}
\usepackage{mathrsfs}
\usepackage{mathtools}
\usepackage{latexsym}
\usepackage{stmaryrd}
\usepackage{epsfig}
\usepackage{amsmath}
\usepackage[usenames,dvipsnames]{xcolor}
\usepackage{bbm}
\usepackage{graphics}
\usepackage[utf8]{inputenc}
\usepackage{amssymb}
\usepackage{amsthm}
\usepackage[alphabetic]{amsrefs}
\usepackage{amsopn}
\usepackage{amscd}
\usepackage[all,knot]{xy}
\xyoption{all}
\usepackage{rotating}
\usepackage{tikz}
\usetikzlibrary{calc}
\usepgflibrary{shapes.geometric}
\usepgflibrary{shapes.misc}
\usetikzlibrary{positioning}
\usetikzlibrary{decorations}
\usetikzlibrary{decorations.pathreplacing}
\usepackage{hyperref}
\usepackage{tikz}
\usepackage{tikz-cd}
\usepackage{microtype} 
\mathtoolsset{showonlyrefs}

\theoremstyle{plain}
\newtheorem{thm}{Theorem}[subsection]
\newtheorem{lem}[thm]{Lemma}
\newtheorem{lemnt}[thm]{False Lemma}
\newtheorem{prop}[thm]{Proposition}
\newtheorem{cor}[thm]{Corollary}

\newtheorem*{thm*}{Theorem}
\newtheorem*{lem*}{Lemma}
\newtheorem*{prop*}{Proposition}
\newtheorem*{cor*}{Corollary}
\newtheorem*{thma*}{Theorem A}
\newtheorem*{thmb*}{Theorem B}

\theoremstyle{definition}
\newtheorem{defn}[thm]{Definition}
\newtheorem*{defn*}{Definition}

\newtheorem{ex}[thm]{Example}
{}
\newtheorem{rem}[thm]{Remark}
\newtheorem*{rem*}{Remark}
\newtheorem*{war*}{Warning}
\newtheorem{war}[thm]{Warning}
\newtheorem{hyp_plain}[thm]{Hypothesis}
\newtheorem*{hyp_plain*}{Hypotheses}
\newtheorem{notation}[thm]{Notation}{}
{}
{}
\newtheorem*{ack}{Acknowledgements}{}
\newtheorem{reminder}[thm]{Reminder}{}
\newtheorem{facts}[thm]{Facts}

\theoremstyle{remark}
{}
{}
{}

\def\to{\longrightarrow} 

\def\Ab{\mathsf{Ab}}

\def\AA{\mathbb{A}}

\def\FF{\mathbb{F}}

\def\NN{\mathbb{N}}

\def\PP{\mathbb{P}}
\def\QQ{\mathbb{Q}}

\def\ZZ{\mathbb{Z}}
\DeclareUnicodeCharacter{221E}{$\infty$}

\def\sfA{\mathsf{A}}
\def\sfB{\mathsf{B}}
\def\sfC{\mathsf{C}}
\def\sfD{\mathsf{D}}

\def\sfI{\mathsf{I}}
\def\sfJ{\mathsf{J}}
\def\sfK{\mathsf{K}}
\def\sfL{\mathsf{L}}

\def\sfN{\mathsf{N}}

\def\sfP{\mathsf{P}}
\def\sfQ{\mathsf{Q}}
\def\sfR{\mathsf{R}}
\def\sfS{\mathsf{S}}
\def\sfT{\mathsf{T}}
\def\sfU{\mathsf{U}}

\def\mcB{\mathcal{B}}

\def\mcM{\mathcal{M}}

\def\mcO{\mathcal{O}}

\def\mcS{\mathscr{S}}
\def\mcT{\mathscr{T}}

\def\mcV{\mathcal{V}}

\def\mcZ{\mathcal{Z}}

\def\mfp{\mathfrak{p}}
\def\mfm{\mathfrak{m}}
\def\mfq{\mathfrak{q}}

\def\op{\mathrm{op}}

\def\fp{\mathrm{fp}}

\def\unit{\mathbf{1}}
\def\Sierp{\mathbf{2}}

\DeclareMathOperator{\colim}{colim}
\DeclareMathOperator{\hocolim}{hocolim}
\DeclareMathOperator{\Hom}{Hom}

\DeclareMathOperator{\im}{im}

\DeclareMathOperator{\modu}{\mathsf{mod}}
\DeclareMathOperator{\Modu}{\mathsf{Mod}}

\DeclareMathOperator{\RHom}{\mathbf{R}Hom}

\DeclareMathOperator{\Coh}{Coh}
\DeclareMathOperator{\QCoh}{QCoh}

\DeclareMathOperator{\Spec}{Spec}

\DeclareMathOperator{\Spc}{Spc}
\DeclareMathOperator{\SPC}{SPC}
\DeclareMathOperator{\Supp}{supp}
\DeclareMathOperator{\supp}{Supp}
\DeclareMathOperator{\SUPP}{SUPP}

\DeclareMathOperator{\loc}{\mathrm{loc}}
\DeclareMathOperator{\thick}{thick}

\DeclareMathOperator{\Ker}{Ker}

\DeclareMathOperator{\serre}{serre}

\DeclareMathOperator{\ssupp}{Supp^\mathrm{s}}
\DeclareMathOperator{\sSupp}{supp^\mathrm{s}}
\DeclareMathOperator{\hsupp}{Supp^\mathrm{h}}
\DeclareMathOperator{\Spcs}{\Spc^\mathrm{s}}

\definecolor{internationalkleinblue}{rgb}{0.0, 0.18, 0.65}

\title{Big categories, big spectra}

\author{Scott Balchin}
\address{Scott Balchin, Max Planck Institute For Mathematics,
Vivatsgasse 7,
53111 Bonn.
Germany}
\email{balchin@mpim-bonn.mpg.de}

\author{Greg Stevenson}
\address{Greg Stevenson, Aarhus University, Department of Mathematics, Ny Munkegade 118, bldg. 1530
DK-8000 Aarhus C, Denmark
}
\email{greg@math.au.dk}

\begin{document}
\begin{abstract}
\noindent 
We introduce a new topological invariant of a rigidly-compactly generated tensor-triangulated category and two new notions of support. The first is based on smashing subcategories: it is unknown whether the frame of smashing subcategories is spatial in general, but supposing it is as in the examples we understand, we show it is equipped with a morphism to the Balmer spectrum which detects the failure of the telescope conjecture and we develop the corresponding support theory. The new invariant, the big spectrum, results from taking the entire collection of localizing ideals seriously and considering prime localizing ideals. Although there are, in principle, a proper class of localizing ideals, we are able to prove the existence of at least one big prime lying over every Balmer prime. We conclude with a pair of examples illustrating our constructions.
\end{abstract}
\thanks{The first author would like to thank the Max Planck Institute for Mathematics for its
hospitality. The second author would like to thank the Hausdorff Research Institute for Mathematics for their support.}

\maketitle

\setcounter{tocdepth}{1}
\tableofcontents





\section{Introduction}

There are, by now, many paths through the landscape of mathematics which pass through the forest of tt-geometry. Information freely travels along well-worn trails between algebraic geometry, topology, representation theory, and the theory of operator algebras. This has been made possible by a concerted effort to domesticate this formerly wild land using tools like the Balmer spectrum of an essentially small tensor-triangulated category (henceforth a `small tt-category') \cite{BaSpec} and Rickard idempotents for rigidly-compactly generated tt-categories \cite{BaRickard} (affectionately known as `big tt-categories').

However, many bastions of wilderness still remain, inhabited by bandersnatches and other such beasts (and no vorpal blade to slay them). For instance, we have made essentially no inroads into understanding big tt-categories where the telescope conjecture fails. Our existing methods are somehow reliant on every coproduct preserving monoidal localization  being a finite localization, i.e.\ corresponding to a localizing ideal generated by compact objects. Indeed, the evidence, e.g.\ \cite{StevensonActions}*{Theorem~7.15} and \cite{BarthelStrat}*{Theorem~9.11}, suggests that approaches based on the spectrum of the compacts are doomed to failure in such cases.

The aim of this paper is to shine a light on some dark corners of the forest. We consider, with a view toward classification problems, the frame of smashing tensor ideals (as studied in \cite{BKSframe} and \cite{Complex}) and the \emph{big spectrum}, which is a totally new invariant. These objects refine the spectrum of the compact objects and are better adapted to cases where the telescope conjecture fails. Our expectation is that the corresponding notions of support can be used to classify localizations in greater generality. We should emphasize that there is a technical difficulty in working with the frame of smashing localizations. Namely, we do not know whether it is spatial, i.e.\ is completely determined by a topological space, in general.

Let us now draw a rough map of the new regions available to us. Any rigidly-compactly generated tt-category $\sfT$ has a set $\mcS(\sfT)$ of \emph{smashing tensor ideals}, i.e.\ localizing ideals whose localization functor is given by tensoring with some idempotent object. These smashing tensor ideals form a lattice under inclusion, and it was shown in \cite{BKSframe} that this lattice is a frame\textemdash{}the motivation was to use tools from pointless topology to study $\sfT$. 

The first main result of this paper is a continuation of this work. We restrict ourselves to the \emph{a priori} special case that this frame is spatial and initiate the study of the corresponding support theory. Stone duality then affords us passage to our first space, the smashing spectrum $\Spcs(\sfT)$, together with the expected functoriality and a natural smashing support $\sSupp$.  This space refines $\Spc (\sfT^c)$, where $\sfT^c$ denotes the subcategory of compact objects, as in the following theorem which summarizes some highlights from Sections \ref{sec:smash} and \ref{sec:comparison}.

\begin{thma*}
Assuming that the frame of smashing ideals is spatial then the space $\Spcs(\sfT)$ has points given by prime smashing ideals. It is equipped with a continuous map
\[
\psi\colon \Spcs(\sfT) \to (\Spc (\sfT^c))^\vee \text{ defined by } \psi(\sfS) = \sfS\cap \sfT^c
\]
where $(\Spc (\sfT^c))^\vee$ is the Hochster dual of $\Spc \sfT^c$. This map is compatible with supports in the sense that $\psi^{-1}\Supp (t) = \sSupp (t)$ for every $t\in \sfT^c$. Moreover, $\psi$ is a homeomorphism if and only if the telescope conjecture holds for tensor ideals.
\end{thma*}

One benefit of the smashing support is that it is inherently defined for arbitrary objects of $\sfT$, with no topological assumptions on the smashing spectrum. Alas, this inherent support is analogous to taking supports of cohomology and so is not well adapted to studying general objects. To address this we introduce in Section~\ref{ssec:small} an analogue of the small support of Foxby \cite{Foxby} (along the lines of the support theory of Balmer and Favi \cite{BaRickard}). This support has good properties when every point of the smashing spectrum is locally closed, which is emerging as a key condition in the current approach to big tt-geometry, and we make a preliminary analysis.

In the original version of this article we made the assertion that the frame of smashing tensor ideals was always spatial. This was based on an incorrect identification of the meet of two subcategories with their intersection; a full account of our sins is given in Appendix~\ref{sec:fail}. It seems to be an important question to understand whether this frame is always spatial or if one needs to consider genuine pointless topology.

Our second and even more tantalizing proposal is to move away entirely from bootstrapping from compact objects and attempt to grapple directly with the whole collection of localizing ideals. This approach has the potential to produce, in full generality, a space analogous to $\Spc (\sfT^c)$ which always classifies radical localizing ideals.

However, one runs into immediate difficulties: the upper bound of a chain of proper localizing subcategories need not be proper, so our friend Zorn forsakes us, and we no longer know we are working with a set of subcategories. Nonetheless, we can consider prime localizing tensor ideals (see Definition~\ref{defn:bigprime}) and the collection $\SPC (\sfT)$ of such, which we call the \emph{big spectrum}. It comes with a notion of support which, in the examples we understand, is sufficiently fine to distinguish big objects and, crucially, is defined without imposing additional topological hypotheses.

While we do not know, in general, that $\SPC (\sfT)$ forms a set, we prove in Theorem~\ref{thm:enoughbigprimes} the existence of big primes for any big tt-category.

\begin{thmb*}
Given $\mathbf{p} \in \Spc (\sfT^c)$ there exists a big prime $\sfP$ such that $\sfP \cap \sfT^c = \mathbf{p}$.
\end{thmb*}

This proceeds via a rather suggestive connection between homological primes and big primes which warrants further exploration and could be relevant in the hunt for tt-residue fields.

The following diagram gives a summary of the spaces that appear in the paper, together with the comparison maps we are able to construct thus far. 
\[
\xymatrix{
\Spc^\mathrm{h}(\sfT^c) \ar[rr]^-\phi \ar[dr]_-\chi && \Spc (\sfT^c) & \\
& \SPC (\sfT) \ar[ur]_-\omega && \Spc^\mathrm{s}(\sfT) \ar[ul]^-\psi
}
\]
But, let us be clear that the diagram comes with a warning label\textemdash{}some maps only exist under certain hypotheses, they are not all continuous as displayed, and we don't know in general that $\Spc^\mathrm{s}(\sfT)$ is a good invariant. As in Theorem~A $\psi$ is only continuous when $\Spc(\sfT^c)$ is viewed with the dual topology, we only know $\omega$ exists when every localizing ideal is radical, and $\chi$ is continuous only for the coarsest of the natural topologies on $\SPC(\sfT)$ (cf.\ Proposition~\ref{prop:bigtosmall} for these last two statements).

We close the paper with two examples demonstrating our offering in action. In Section~\ref{sec:ringex} we return to a familiar friend, the derived category of a commutative noetherian ring. The telescope conjecture holds, so the smashing spectrum reduces to the usual one. We describe the smashing primes and prove that the refined smashing support recovers Foxby's small support. A complete description of the big spectrum is also given\textemdash{}again it is in natural bijection with points of the corresponding affine scheme and the associated support theory (which we emphasize again is defined without topological restriction) recovers the small support.

In the final section we examine a case where the telescope conjecture fails. We see concretely that the smashing spectrum is a finer invariant than the homological and Balmer spectra, and that the refined smashing support detects more than its counterpart valued in the spectrum of the compacts.

\begin{ack}
We are indebted to Paul Balmer for several valuable suggestions which greatly improved the exposition. Our thanks go out to Jan \v{S}\v{t}ov\'{\i}\v{c}ek for precious discussions (see Remark~\ref{ex:notacohframe}). We also thank Charalampos Verasdanis for helpful comments, and the anonymous referees for many helpful suggestions leading to a significantly improved exposition and the inclusion of some additional statements. Finally, we are grateful to Peter Scholze for questions which led to the identification of a significant error in the previous version.
\end{ack}




\section{Preliminaries}\label{sec:prelim}

\begin{hyp_plain*}
Throughout $(\sfT,\otimes,\unit)$ will denote a rigidly-compactly generated tensor-triangulated category. In particular, $\sfT$ admits all small coproducts, the compact and rigid objects coincide, and they moreover form an essentially small subcategory $\sfT^c$ which generates $\sfT$ as a localizing category. We will refer to such a $\sfT$ as a \emph{big tt-category}.
\end{hyp_plain*}

\begin{war*}
Throughout we shall denote by $\Supp (-)$, with appropriate decoration, the `native' support theory associated to some space, e.g.\ Balmer's support of compact objects or the (unrefined) smashing support of Definition~\ref{defn:ssupp}. We reserve the notation $\supp(-)$, and its variants, for the refinement of this support in terms of certain idempotent objects, and call it the \emph{small support}. This agrees with the conventions of Balmer--Favi~\cite{BaRickard} which have propagated throughout the literature on tt-categories. 

The warning is that this notation is in conflict with the usual notation for the sheaf-theoretic support and for the small support of Foxby for bounded complexes of flat modules~\cite{Foxby} (it is precisely reversed: over a ring $R$ the module support $\supp_R$ is refined by the small support $\Supp_R$). This is quite unfortunate from a story-telling point of view, but we stick to the abstract notions of support which provides adequate protection against confusion.
\end{war*}

\subsection{Stone duality}~\label{ssec:stone}

First off, let us give some definitions from lattice theory, and a quick summary of Stone duality which allows us to move between sober topological spaces and spatial frames without loss of information. We don't give examples here, but rather we spell out what these constructions offer us in practice throughout the sequel. Useful references are \cite{Johnstone} and \cite{SpectralBook} where the proofs we omit here (i.e.\ all of them) can be found.

Let $L$ be a complete lattice, i.e.\ a poset which admits an infimum and supremum for any subset of elements. We call these the meet and join respectively, and denote them by $\wedge$ and $\vee$. In this section we denote the minimal element of $L$ by $0$ and the maximal element by $1$. 

\begin{defn}
The complete lattice $L$ is a \emph{frame} if for any $a \in L$ and $\{b_i\mid i\in I\}\subseteq L$ we have
\[
a\wedge (\bigvee_{i\in I} b_i) = \bigvee_{i\in I} (a\wedge b_i),
\]
that is finite meets distribute over arbitrary joins. Given frames $L$ and $L'$ a morphism $f\colon L\to L'$ is a poset map which preserves arbitrary joins and finite meets.
\end{defn}

We denote by $\Sierp$ the frame $\{0\leq 1\}$, which should be thought of as the analogue of the topological space consisting of just a single point, and plays a distinguished role. 

\begin{defn}
A \emph{point} of a frame $L$ is a map $p\colon L\to \Sierp$. We define the \emph{spectrum} of $L$ to be
\[
\Spec (L) = \{p\colon L\to \Sierp\}
\]
and topologize this by declaring the subsets 
\[
U_a = \{p\in \Spec (L)\mid p(a) = 1\},
\]
indexed by elements $a\in L$, to be open. We denote the closed complement of $U_a$ by $V_a$. This is evidently contravariantly functorial: we have $\Spec = \Hom(-, \Sierp)$.
\end{defn}

We say that a frame $L$ is \emph{spatial} if whenever $a\nleq b$ in $L$ there is a point $p$ such that $p(a)=1$ and $p(b)=0$, i.e.\ $p$ witnesses that $a\nleq b$. The name is due to the fact that, as we will see below, a frame is spatial if and only if it is isomorphic to the lattice of open subsets of a topological space.

\begin{defn}\label{defn:cohframe}
An element $a\in L$ is \emph{finitely presented} if whenever $a\leq \vee \{b_i\mid i\in I\}$ there are $b_1,\ldots,b_n \in \{b_i\mid i\in I\}$ such that $a\leq b_1\vee \cdots \vee b_n$.

We say $L$ is a \emph{coherent frame} if the finitely presented elements of $L$ form a sublattice containing $0$ and $1$ (really this means $1$ is finitely presented and the finitely presented elements are closed under finite meets) and every element of $L$ is a join of finitely presented elements. It is automatic that a coherent frame is spatial.
\end{defn}

We now turn to the topological analogues. Recall that a space $X$ is \emph{sober} if every non-empty irreducible closed subset of $X$ has a unique generic point. A space $X$ is \emph{spectral} if it is sober, quasi-compact, and has a basis of quasi-compact open subsets closed under finite intersections. Given a space $X$ we denote by $\mcO(X)$ the lattice of open subsets of $X$, which is easily seen to be a frame. For a map $f\colon X\to Y$ we get $f^{-1}\colon \mcO(Y) \to \mcO(X)$ making $\mcO$ into a contravariant functor.

Stone duality is the assertion that sober and spectral spaces are the topological avatars of spatial and coherent frames respectively.

\begin{thm}\label{thm:stoneduality}
Let $L$ be a frame and $X$ a topological space. Then $\Spec (L)$ is sober and $\mcO(X)$ is a frame. In particular, we can restrict $\Spec$ and $\mcO$ to the category of spatial frames, with frame maps, and the category of sober spaces with continuous maps. Thus restricted the functors $\Spec$ and $\mcO$ are inverse to one another, i.e.\ they give a duality between spatial frames and sober spaces. Under this equivalence the coherent frames correspond to the spectral spaces.
\end{thm}
\begin{proof}
See \cite{Johnstone}*{Chapter~II} or \cite{SpectralBook}*{3.2.10}.
\end{proof}

Let $L$ be a spatial frame. The penultimate circle of ideas we will need concerns how to describe points of $\Spec (L)$ in terms of elements of $L$. Let $p\colon L\to \Sierp$ be a point. We can consider
\[
P = \bigvee \{a\in L \mid p(a) = 0\}.
\]
Since $p$ commutes with joins $p(P) = 0$, and so $P$ is the (unique) maximal element which $p$ sends to $0$. If $a\wedge b \leq P$ then $p(a) \wedge p(b) = 0$ and so one of $p(a)$ or $p(b)$ must be $0$, i.e.\ $p(a)\leq P$ or $p(b)\leq P$. Thus $P$ is \emph{meet-prime}: if $a\wedge b \leq P$ then $a\leq P$ or $b\leq P$. This sets up a bijection between points of $L$ and meet-prime elements of $L$ (see \cite{Johnstone}*{II.1.3}). Indeed, given a meet-prime $P$ the corresponding point $p$ is defined by
\[
p(a) = \begin{cases}
0  & \text{ if } a\leq P \\
1  & \text{ if } a\nleq P.
\end{cases}
\]

Let us summarize the key pieces we will need from the above discussion:

\begin{facts}\label{justthefactsmaam}
Let $L$ be a spatial frame.
\begin{itemize}
\item[(i)] $\Spec (L)$ is a sober space whose open and closed subsets are precisely the
\[
U_a = \{p \in \Spec (L)\mid p(a)=1\} \text{ and } V_a = \{p\in \Spec (L)\mid p(a)=0\}
\]
respectively.
\item[(ii)] $L$ is isomorphic as a lattice to $\mcO(\Spec (L))$, the lattice of open subsets of $\Spec (L)$ (with no prize for guessing the map). 
\item[(iii)] The points of $\Spec (L)$ are in bijection with the meet-prime elements of $L$. Viewing this as an identification we have
\[
U_a = \{P \in \Spec (L)\mid a\nleq P\} \text{ and } V_a = \{P\in \Spec (L)\mid a\leq P\}.
\]
\end{itemize}
\end{facts}

The final significant piece of lattice-theoretic technology we will need is that of \emph{Hochster duality} for spectral spaces \cite{HochsterSpectral}. Given a spectral space $X$ the Hochster dual $X^\vee$ has the same points as $X$ and the topology whose opens are generated by the closed subsets of $X$ whose complements are quasi-compact. In other words, the open subsets of $X^\vee$ are the Thomason subsets of $X$. Hochster proved that the space $X^\vee$ is again spectral and $(X^\vee)^\vee \cong X$. Note that \cite[1.4]{SpectralBook} refers to the Hochster dual topology as the \emph{inverse topology}.

Next, let us briefly recapitulate the two current topological invariants that we have available to us when working with a big tt-category, namely the \emph{Balmer} and \emph{homological} spectra.

\subsection{The Balmer spectrum}

Associated to any big tt-category $\sfT$ we have the \emph{Balmer spectrum} $\Spc(\sfT^c)$ which is the set of prime thick tensor ideals of $\sfT^c$ (i.e.\ proper thick subcategories closed under tensoring with arbitrary objects of $\sfT^c$ such that if $a \otimes b \in \sfP$ then $a \in \sfP$ or $b \in \sfP$)~\cite{BaSpec}. For $x \in \sfT^c$ the support of $x$ is
\[
\Supp(x) = \{\sfP \in \Spc(\sfT^c) \mid x \not\in \sfP \},
\]
and we take these as a basis of closed subsets for $\Spc(\sfT^c)$. A key point is that the Balmer spectrum is in fact a spectral space.  Under certain technical hypotheses one can extend the support to a theory valid for all objects of $\sfT$, which we call the small support and denote by $\supp$. We refer the reader to~\cite[Section 7]{BaRickard} for details.

More relevant to the material in the present paper will be the lattice-theoretic point of view. Denote by $\mcT(\sfT^c)$ the lattice of thick tensor ideals of $\sfT^c$. This lattice is a frame, which is moreover coherent~\cite[Theorem 3.1.9]{KockPitsch} and so, \emph{a fortiori} it is spatial, i.e.\ isomorphic to the lattice of open subsets of some space. 

Stone duality then gives rise to the spectral space $\Spec(\mcT(\sfT^c))$ whose points we can identify with the meet-prime elements of $\mcT(\sfT^c)$ and whose lattice of opens is isomorphic to $\mcT(\sfT^c)$. (An excellent overview is given in \cite{SpectralBook}*{Chapters~2 and 3} and we briefly recalled aspects of the construction in Section~\ref{ssec:stone}.)  The primes here coincide with those introduced by Balmer: there is a canonical bijection $\Spec(\mcT(\sfT^c)) \cong \Spc(\sfT^c)$. To relate the topologies, one first takes the Hochster dual of $\Spec(\mcT(\sfT^c))$, denoted $\Spec(\mcT(\sfT^c))^\vee$, whose open subsets are generated by the closed subsets of $\Spec(\mcT(\sfT^c))$ with quasi-compact open complements (see for example~\cite{SpectralBook}). By~\cite[Corollary 3.4.2]{KockPitsch} there is then a homeomorphism $\Spec(\mcT(\sfT^c))^\vee \cong \Spc(\sfT^c)$.

\subsection{The homological spectrum}

We denote by $\Modu \sfT^c$ the category of additive functors $(\sfT^c)^{\op} \to \Ab$. This is a Grothendieck category and is equipped, via Day convolution, with a closed symmetric monoidal structure. We denote by $\modu \sfT^c$ the (abelian) subcategory of finitely presented objects. We have the usual Yoneda embedding $h \colon \sfT^c \hookrightarrow \modu \sfT^c$ which can be extended to the \emph{restricted Yoneda functor} $h \colon \sfT \to \Modu \sfT^c$ by mapping $x \in \sfT$ to $h(x) = \hat{x}$ via
\[
\hat{x} \colon (\sfT^c)^{\op} \to \Ab \qquad \qquad \hat{x}(-) = \Hom_{\sfT}(-,x) {\mid}_{\sfT^c}.
\]
This restricted Yoneda is a conservative, (co)product preserving, and symmetric monoidal homological functor which fits into a commutative diagram
\[
\xymatrix{
\sfT^c \ar@{^(->}[d] \ar[r]^-h & \modu \sfT^c \ar@{^(->}[d] \\
\sfT \ar[r]_-h & \Modu \sfT^c \rlap{ .}
}
\]
Here $\modu \sfT^c$ is a monoidal subcategory of $\Modu \sfT^c$, and the tensor product on $\modu \sfT^c$ determines the one on the larger category by taking the unique colimit compatible extension. We refer the reader to the appendix of~\cite{BKSframe} and to \cite{BKSfield} for more details on this construction and the module category in general.

A \emph{homological prime} for $\sfT^c$ is defined to be a maximal proper Serre $\otimes$-ideal $\mcB \subset \modu \sfT^c$. The \emph{homological residue field}, $\overline{h}_{\mcB}$, corresponding to a homological prime $\mcB$ is constructed via the composite
\[\xymatrix@R=0em{
 \overline{h}_{\mcB} = Q_{\mcB} \circ h :& \sfT \ar[r]^-{h} & \Modu \sfT^c \ar@{->>}[r]^{Q_{\mcB}} & \frac{\Modu \sfT^c}{ \mcB^\shortrightarrow} \\
&x \ar@{|->}[r] & \hat{x} \ar@{|->}[r] & \overline{x} \;\; .
}\]
Here, $h \colon \sfT \to \Modu \sfT^c$ is the restricted Yoneda as defined above, $ \mcB^\shortrightarrow$ is the localizing subcategory in $\Modu \sfT^c$ generated by $\mcB$, and the quotient is the usual Gabriel quotient.

By assumption $\sfT^c$ is essentially small. Hence $\modu \sfT^c$ is also essentially small and so admits only a set of Serre subcategories. In particular, there is only a set of homological primes which we denote by $\Spc^\mathrm{h}(\sfT^c)$:
\[
\Spc^\mathrm{h}(\sfT^c) = \{\mcB \subset \modu \sfT^c \mid \mcB \text{ is a homological prime}\}.
\]
We equip this set with the topology generated by the closed subsets $\Supp^\mathrm{h} (x)$ for all $x \in \sfT^c$:
\[
\Supp^\mathrm{h}(x) := \{\mcB \in \Spc^\mathrm{h}(\sfT^c)  \mid \hat{x} \not\in \mcB \} = \left\{\mcB \in \Spc^\mathrm{h}(\sfT^c)  \mid \overline{x}\neq 0 \text{ in } \frac{\Modu \sfT^c}{ \mcB^\shortrightarrow  }  \right\}.
\]
We call $\Spc^\mathrm{h}(\sfT^c)$ along with this topology the \emph{homological spectrum}.

For every $\mcB \in \Spc^\mathrm{h}(\sfT^c)$ its preimage in $\sfT^c$ under the Yoneda embedding 
\[
\phi (\mcB) = h^{-1}(\mcB) \cap \sfT^c = \{x \in \sfT^c \mid \hat{x} \in \mcB \}
\]
is a Balmer prime ideal in $\sfT^c$. As such, there is a surjective (and continuous) map $\phi \colon \Spc^\mathrm{h}(\sfT^c) \to \Spc(\sfT^c)$~\cite[3.1]{BaHomological}. This can also be seen by using the fact that the Balmer support of the compact objects is the terminal support theory~\cite[Theorem 3.2]{BaSpec}.

A lattice theoretic interpretation of the homological spectrum does not appear in the literature, and it is unclear if such a perspective can exist. Indeed, it is not even known if the homological spectrum is $T_0$ in general (if it is $T_0$, then it will be homeomorphic to $\Spc(\sfT^c)$ via the comparison map $\phi$~\cite{BHS2021}).



\section{The smashing spectrum}\label{sec:smash}
\setcounter{subsection}{1}

Throughout this section we will make extensive use of the yoga of smashing ideals and their associated idempotents. We refer to \cite{BaRickard} for the core techniques and to \cite{BKSframe} for preliminaries on right (and dually left) idempotents in a symmetric monoidal category.

Let $\sfT$ be a big tt-category. We denote by $\mcS(\sfT)$ the collection of smashing tensor ideals. This collection forms a set (this was originally proved by Krause \cite{KrTele}) and is naturally ordered by inclusion. It is a complete lattice where finite meets are given by intersection and joins are given, as usual for such lattices, by taking the localizing subcategory generated by the union (see \cite{BKSframe}*{Remark~5.12} and Lemma~\ref{lem:join} for justifications). It was shown in \cite{BKSframe} that $\mcS(\sfT)$ is not only a complete lattice but a frame i.e.\ finite meets distribute over arbitrary joins.

\begin{hyp_plain}\label{thm:spatial}
From this point onward we will assume that $\mcS(\sfT)$ is spatial.
\end{hyp_plain}

\begin{rem}
Perhaps this is for free, or perhaps this is a serious restriction\textemdash{}we do not know (anymore). Even if it proves to be a restriction it is justified by the wealth of interesting examples coming from derived categories of valuation domains by work of Bazzoni and \v{S}\v{t}ov\'{\i}\v{c}ek \cite{BazzoniStovicek}*{Theorem~5.23}.
\end{rem}

\begin{reminder}\label{rem:isos}

We remind ourselves of the following diagram of isomorphisms from \cite{BKSframe}*{Remark~4.9}
\[
\kern-.5em
\xymatrix@C=4em@R=2em{
\kern-1em\left\{{{\displaystyle\textrm{smashing $\otimes$-ideals}}
 \atop{\vphantom{I^{I^I}}\displaystyle \sfS\subseteq\sfT}}\right\}
   \ar[d]_-{\textrm{\cite{BaRickard}}}
&
\left\{{{\displaystyle\textrm{localiz.\,$\otimes$-ideals $\mcB\subseteq\Modu \sfT^c$ s.t.\ }\mcB=(\mcB^\fp)^\shortrightarrow \textrm{ and}}
 \atop{\vphantom{I^{I^I}}\displaystyle \textrm{s.t.\ }\mathscr{J}=\{f\textrm{ in }\sfT^c \mid \im(\hat f)\in\mcB\}\textrm{ equals }\mathscr{J}^2}}\right\}
 \ar[l]_-{h^{-1}(\mcB)\, \mapsfrom \,\mcB}
\\
\left\{{{\displaystyle\textrm{right-idempotents }}
 \atop{\vphantom{I^{I^I}}\displaystyle \unit \to F \ \textrm{ in }\sfT}}\right\}_{\!\!/\simeq}
 \ar[r]_-{h}
&
\left\{{{\displaystyle\textrm{flat right-idempotents }}
 \atop{\vphantom{I^{I^I}}\displaystyle\unit \to F \textrm{ in }\Modu \sfT^c}}\right\}_{\!\!/\simeq}
  \ar[u]_-{[F]_\simeq\,\mapsto \,\Ker(F\otimes-)}
}
\]
Here $\mcB^\fp$ denotes the full subcategory of finitely presented objects of $\mcB$ and $(\mcB^\fp)^\shortrightarrow$ denotes its completion under colimits in $\Modu \sfT^c$, i.e.\ the smallest localizing subcategory containing $\mcB^\fp$.

These isomorphisms all respect the orderings and so are lattice isomorphisms. As such we will regularly abuse notation and identify idempotents with their image in the module category. For a right idempotent $F$ we set $\Ker(h(F)\otimes-) = \mcB_F$ as in the righthand vertical map above.
\end{reminder}

\subsection{The smashing support}

Our starting point, based on Hypothesis~\ref{thm:spatial}, is the following definition.

\begin{defn}\label{defn:sspec}
It follows from Hypothesis~\ref{thm:spatial} that there is a sober space $\Spc^\mathrm{s}(\sfT)$, the \emph{smashing spectrum of} $\sfT$, which is dual to $\mcS(\sfT)$ via Stone duality.
\end{defn}

\begin{rem}\label{rem:openclosed}
Let us take some time to unpack what this means. By Fact~\ref{justthefactsmaam}(iii) the points of $\Spc^\mathrm{s}(\sfT)$ are in bijection with meet-prime smashing ideals of $\sfT$, and we take this as the primitive notion (it will usually be more convenient for us than working with lattice morphisms to $\Sierp$). By Fact (ii) the lattice of open subsets of $\Spc^\mathrm{s}(\sfT)$ is isomorphic to $\mcS(\sfT)$. Given a smashing subcategory $\sfS$ the corresponding open subset of $\Spc^\mathrm{s}(\sfT)$ is
\[
U_\sfS = \{\sfP \in \Spc^\mathrm{s}(\sfT) \mid \sfS \nsubseteq \sfP\}
\]
and its closed complement is 
\[
V_\sfS = \{\sfP \in \Spc^\mathrm{s}(\sfT) \mid \sfS \subseteq \sfP\}.
\]
Let us emphasize that the closure of $\sfP \in \Spc^\mathrm{s}(\sfT)$, namely $V_\sfP$, is given by the meet-prime smashing ideals \emph{containing} $\sfP$ as in algebraic geometry (foreshadowing the order reversal of Theorem~\ref{thm:comparison}).
\end{rem}

\begin{rem}
The previous remark may cause some alarm\textemdash{}as well it should since there is an opportunity for confusion. However, let us reassure the reader that almost everything we do is very formal and is based on manipulating open and closed subsets and idempotents in a way that is insensitive to their precise form. The major difference from traditional tt-geometry is rather cosmetic, and consists of replacing Thomason subsets by opens.

Perhaps the most aggressive appearance of this reversal occurs in Lemma~\ref{lem:reverseit}. One sees that, even there, the phenomenon is quite mild and involves reasoning with familiar subsets in familiar ways.
\end{rem}

Going through the usual motions we arrive at our first new notion of support.
 
\begin{defn}\label{defn:ssupp}
The \emph{smashing support} of an object $X\in \sfT$ is
\[
\sSupp (X) = \{\sfP \in \Spc^\mathrm{s}(\sfT) \mid X\notin \sfP\}.
\]
\end{defn}

\begin{rem}
Let $x\in \sfT^c$ be a compact object. Then $\sfS = \loc^\otimes(x)$ is a smashing ideal. We have just unpacked that the corresponding open subset of $\Spc^\mathrm{s}(\sfT)$ is
\[
U_\sfS = \{\sfP \in \Spc^\mathrm{s}(\sfT) \mid \sfS \nsubseteq \sfP\} = \{\sfP \in \Spc^\mathrm{s}(\sfT) \mid x\notin \sfP\}.
\]
Thus it is natural to define the support of $x$ to be this open subset. Definition~\ref{defn:ssupp} is motivated by extending this observation to arbitrary objects.
\end{rem}

We note that the smashing support is always generization closed, simply because open subsets are so (this is analogous to specialization closure of the supports of compacts in $\Spc(\sfT^c)$).

\begin{war}
In the case of the Balmer spectrum, the topology is generated by the supports of compact objects. This is not necessarily the case for the smashing spectrum. We expect trouble when the telescope conjecture fails for $\sfT$, and we find it; we refer the reader to Remark~\ref{rem:notsmalltop} for an explicit example.
\end{war}

\begin{rem}
Denoting by $F_\sfP$ the right idempotent associated to a meet-prime smashing ideal $\sfP$ we have
\[
\sSupp (X) = \{\sfP \in \Spc^\mathrm{s}(\sfT) \mid X\otimes F_\sfP \neq 0\}.
\]
\end{rem}

\begin{lem}\label{lem:ssupp}
The smashing support satisfies the following properties:
\begin{itemize}
\item[(1)] $\sSupp(0) = \varnothing$ and $\sSupp(\unit) = \Spc^\mathrm{s}(\sfT)$;
\item[(2)] $\sSupp (\coprod_i X_i) = \cup_i \sSupp (X_i)$;
\item[(3)] $\sSupp (\Sigma X) = \sSupp (X)$;
\item[(4)] for any triangle $X \to Y \to Z$ we have $\sSupp (Y) \subseteq \sSupp (X) \cup \sSupp (Z)$;
\item[(5)] $\sSupp (X\otimes Y) \subseteq \sSupp (X)\cap \sSupp (Y)$;
\item[(6)] if $x,y\in \sfT^c$ then $\sSupp (x \otimes y) = \sSupp (x)\cap \sSupp (y)$.
\end{itemize}
\end{lem}
\begin{proof}
Properties (1)-(4) follow in a straightforward way from the fact that prime smashing ideals are localizing subcategories. Property (5) is an immediate consequence of the fact that they are ideals.

For (6) we need to show the reverse inclusion of (5). Since $x$ and $y$ are compact they generate smashing ideals $\sfS_x = \loc^\otimes(x)$ and $\sfS_y = \loc^\otimes(y)$. If $\sfP \in \sSupp x\cap \sSupp y$, i.e.\ $x\notin \sfP$ and $y\notin \sfP$ then $\sfS_x \nsubseteq \sfP$ and $\sfS_y \nsubseteq \sfP$. As $\sfP$ is meet-prime it follows that $\sfS_x\cap \sfS_y \nsubseteq \sfP$, and the claim then follows from Lemma~\ref{lem:tensor} which tells us that $\sfS_x \cap \sfS_y = \sfS_{x\otimes y}$ so $x\otimes y\notin \sfP$.
\end{proof}

\begin{rem}
Property (6) cannot be generalized to a `half $\otimes$-formula' where only one of the objects is asked to be compact. Consider $\sfT = \sfD(\ZZ)$ and take the smashing prime $\sfP$ given by the kernel of the localization to $\sfD(\ZZ_{(p)})$ (cf.\ Section~\ref{ssec:sp}). Then $\QQ$ and $\FF_p$ do not lie in $\sfP$ but their tensor product, which is $0$, does. However, this does hold when one can refine the support\textemdash{}see Proposition~\ref{prop:half}.
\end{rem}

Given localizing ideals $\sfL_1$ and $\sfL_2$ of $\sfT$ we denote by $\sfL_1\otimes \sfL_2$ the localizing ideal generated by the objects $L_1\otimes L_2$ with $L_i \in \sfL_i$. We will use identical notation, replacing localizing by thick, when working with thick ideals in $\sfT^c$ and rely on context to distinguish these.

\begin{lem}\label{lem:tensor}
If $\sfS$ and $\sfR$ are smashing ideals then $\sfS\cap \sfR = \sfS\otimes \sfR$. Moreover, if $\sfS = \loc(\sfS^c)$ and $\sfR = \loc(\sfR^c)$ then $\sfS \cap \sfR = \loc (\sfS^c \cap \sfR^c) = \loc(\sfS^c \otimes \sfR^c)$.
\end{lem}
\begin{proof}
Let $E_\sfS$ and $E_\sfR$ denote the corresponding left idempotents, so that $\sfS = \loc^\otimes(E_\sfS)$ and similarly for $\sfR$. Then, by \cite{StevensonActions}*{Lemma~3.12} we have
\[
\sfS \otimes \sfR = \loc^\otimes(E_\sfS \otimes E_\sfR)
\]
and the latter is equal to $\sfS\cap \sfR$.

For the second statement suppose that $\sfS$ and $\sfR$ are generated by compact objects. Then, again using \cite{StevensonActions}*{Lemma~3.12}, we have
\[
\sfS \cap\sfR = \sfS \otimes \sfR = \loc^\otimes(\sfS^c) \otimes \loc^\otimes(\sfR^c) = \loc^\otimes(\sfS^c \otimes \sfR^c). \qedhere
\]
\end{proof}

\begin{rem}
Recall that we are assuming that $\sfT^c$ is rigid and so $\sfS^c \otimes \sfR^c = \sfS^c \cap \sfR^c$.
\end{rem}

\begin{rem}\label{ex:notacohframe}
Under the assumption that $\mcS(\sfT)$ is a spatial frame we know $\Spc^\mathrm{s}(\sfT)$ is sober. It is natural to ask when $\mcS(\sfT)$ is a coherent frame, i.e.\ when $\Spc^\mathrm{s}(\sfT)$ is a spectral space. This, it seems, is not for free.

Bazzoni and \v{S}\v{t}ov\'{\i}\v{c}ek have computed the frame $\mcS(\sfD(R))$ for $R$ a valuation domain \cite{BazzoniStovicek}*{Theorem~5.23}. It is given by certain subchains of the poset of intervals $[\mathfrak{i}, \mfp]$ in $\Spec R$ where $\mathfrak{i}^2 = \mathfrak{i}$. In \cite[Example 5.1]{Bazzoni15} Bazzoni gives an existence result for valuation domains with spectrum of a rather particular form. One needs to control exactly which prime ideals can be idempotent, but it should be possible to massage this into an example of a valuation domain such that the compact objects of $\mcS(\sfD(R))$ are exactly the singletons $\{[0,\mathfrak{i}]\}$ with $\mathfrak{i}^2 = \mathfrak{i}$. There are not enough of these intervals to form a basis of quasi-compact opens for the resulting topological space. 

In particular, this contradicts an assertion made in an earlier version of this article. The problematic point was that it is \emph{not} the case that an intersection of exact ideals of $\modu \sfT^c$ is necessarily an exact ideal (cf.\ the counterexamples in Section~\ref{sec:fail}). Hence one cannot guarantee the existence of perfect Serre tensor ideals generated by a set of finitely presented $\sfT^c$-modules. We are very grateful to Jan \v{S}\v{t}ov\'{\i}\v{c}ek for pointing out Bazzoni's example and explaining its relevance to us, and for Peter Scholze for a penetrating question on this issue. Both of these helped us to realize our reasoning had a gap, and to locate the issues.
\end{rem}


\subsection{Small smashing support}\label{ssec:small}

The notion of support defined in the previous section, while natural from the point of pointless topology, is not refined enough for effectively dealing with non-compact objects (cf.\ Lemma~\ref{lem:exsupp}). Since the dream is, at least in some examples (but unfortunately not all, cf.\ Remark~\ref{rem:fail}), to capture general localizing phenomena we need an analogue of the small support of Foxby~\cite{Foxby} (as extended in the work of several others, \cite{BaRickard} being most relevant to us here). 

We will end up using a similar approach to that of~\cite{BaRickard}. Given a smashing ideal $\sfS$ let us denote by $E_\sfS$ and $F_\sfS$ the associated left and right idempotents. 

Fix a meet-prime smashing ideal $\sfP$. For $X\in \sfT$ considering $X\otimes F_\sfP$ cuts us down to $\sfT/\sfP$ i.e.\ to considering the smashing ideals in the interval $[\sfP,\sfT]$. The primes in this interval give the closure $V_\sfP$ of $\sfP$ in $\Spc^\mathrm{s}(\sfT)$ (cf.\ Remark~\ref{rem:openclosed}). The problem is that simply looking here won't detect if $X$ really lives `further down' because if $\sfP\subseteq \sfQ$ and $X\otimes F_\sfQ\neq 0$ then necessarily this is the case for $X\otimes F_\sfP$. In other words, perhaps $\sfP \in \sSupp (X)$ is an artifact of $\sfQ \in \sSupp (X)$ and not immanent.

We thus want to ensure that for $\sfP$ to be in the support of $X$ that $X$ is `essential' in $\sfT/\sfP$, i.e.\ it doesn't come from a further localization. As in other versions of big tt-geometry we need hypotheses on the spectrum to get our way.

\begin{rem}
We have seen that $\mcS(\sfT)$ is likely not generally a coherent frame (Example~\ref{ex:notacohframe}) even when it is spatial as we are assuming. As such, we cannot take the Hochster dual to work with Thomason subsets as is traditional e.g.\ for $\Spc (\sfT^c)$. Consequently, let us reiterate that we will be interested in open subsets as opposed to the Thomason subsets which are used in~\cite{BaRickard}.
\end{rem}


\begin{defn}\label{defn:gamma}
Suppose that $\sfP$, viewed as a point of $\Spc^\mathrm{s}(\sfT)$, can be written in the form $\{\sfP\} = U_\sfS \cap V_\sfP$ for some open subset $U_\sfS$ (recall every open has this form\textemdash{}see Remark~\ref{rem:openclosed}). We define the \emph{Rickard idempotent at} $\sfP$ by 
\[
\Gamma_\sfP = E_\sfS \otimes F_\sfP.
\]
\end{defn}

\begin{rem}
A consequence of the above assumption is that $\{\sfP\}$ is locally closed (i.e. can be written as the intersection of an open and a closed subset). As we will show in Lemma~\ref{lem:Fp} the condition that $\{\sfP\} = U_\sfS \cap V_\sfP$ is in fact equivalent to $\{\sfP\}$ being locally closed.
\end{rem}

\begin{rem}
Again let us unwind what is happening here, using the description of the opens and closed subsets as in Remark~\ref{rem:openclosed}. We have 
\[
V_\sfP = \{\sfQ \in \Spc^\mathrm{s}(\sfT) \mid \sfP \subseteq \sfQ\} \text{ and } U_\sfS = \{\sfQ \in \Spc^\mathrm{s}(\sfT) \mid \sfS \nsubseteq \sfQ\}.
\]
For $\sfP$ to be locally closed we are asking that there is an $\sfS$ such that $\sfS \nsubseteq \sfP$ and such that $\sfP$ is maximal with this property. Tensoring with $F_\sfP$ lands us in $\sfT/\sfP$, where the image of $\sfS$ is not zero and hence the image of $E_\sfS$ is non-zero. Localizing at a bigger meet-prime smashing ideal, i.e.\ passing to $\sfT/\sfQ$ for $\sfQ\in V_\sfP \setminus \{\sfP\}$ would kill $E_\sfS$ and so we are isolating, via $\Gamma_\sfP = E_\sfS \otimes F_\sfP$, some information local to $\sfP$.
\end{rem}

\begin{rem}\label{rem:gammapnotzero}
The object $\Gamma_\sfP$ is not zero and we will prove it. Write $\{\sfP\} = U_\sfS \cap V_\sfP$ so $\Gamma_\sfP = E_\sfS \otimes F_\sfP$. This is zero if and only if $E_\sfS \in \sfP$ which is equivalent to $\sfS\subseteq \sfP$. We have explicitly chosen $\sfS$ so that this is not the case.
\end{rem}

Thus we want to work with spaces all of whose points are locally closed. This property turns out to be a separation axiom, lying between $T_0$ and $T_1$, which was originally studied in \cite{AullThron}.

\begin{defn}
A space $A$ is $T_D$ if every point of $A$ is locally closed.
\end{defn}

\begin{defn}\label{defn:smallsupport}
Suppose that $\Spc^\mathrm{s}(\sfT)$ is $T_D$. Then the \emph{small smashing support} of an object $X\in \sfT$ is
\[
\ssupp (X) = \{\sfP \in \Spc^\mathrm{s}(\sfT) \mid \Gamma_\sfP \otimes X \neq 0\}.
\]
\end{defn}

\begin{rem}
It is clear from the definition that the analogues of Lemma~\ref{lem:ssupp} (1)-(5) hold for the small smashing support.
\end{rem}

The notion of support we get doesn't depend on any choices we made regarding the $\Gamma_\sfP$.

\begin{lem}\label{lem:indep}
Let $\sfA, \sfB, \sfC$ and $\sfD$ be smashing ideals. If  
\[
U_\sfA \cap V_\sfB = U_\sfC \cap V_\sfD \text{ then } E_\sfA \otimes F_\sfB \cong E_\sfC \otimes F_\sfD.
\]
\end{lem}
\begin{proof}
We can split the proof into steps by first varying the open set and then varying the closed set. The proof of  \cite{BaRickard}*{Lemma~7.4} \emph{mutatis mutandis} applies to show these statements hold (their proof refers to two other results in their paper, but one is a general statement about idempotents and the other is baked into our setting so there is no problem).
\end{proof}

Moreover, if $\sfP$ is locally closed so one can isolate $\{\sfP\}$ via a closed and open subset then we can always assume that the closed subset involved is $V_\sfP$, so our definition of $\Gamma_\sfP$ using $F_\sfP$ is not restrictive.

\begin{lem}\label{lem:Fp}
Suppose that $\sfS$ and $\sfR$ are smashing ideals such that $U_\sfS \cap V_\sfR = \{\sfP\}$ i.e.\ $\sfP$ is locally closed. Then $U_\sfS \cap V_\sfP = \{\sfP\}$.
\end{lem}
\begin{proof}
Let $\sfR$ and $\sfS$ be as in the statement. We have $\sfP \in V_\sfR$ and hence the closure of $\sfP$ is contained in $V_\sfR$. The closure of $\sfP$ is precisely $V_\sfP$, so $V_\sfP \subseteq V_\sfR$ and it follows that $U_\sfS \cap V_\sfP = \{\sfP\}$.
\end{proof}

\begin{lem}\label{lem:bigvssmallsupp}
Suppose that $\Spc^\mathrm{s}(\sfT)$ is $T_D$. For any $X\in \sfT$ we have
\[
\ssupp (X) \subseteq \sSupp (X)
\]
with equality if $X$ is compact.
\end{lem}
\begin{proof}
Suppose that $\sfP \in \ssupp (X)$. Then necessarily $F_\sfP \otimes X\neq 0$ and so $\sfP\in \sSupp (X)$.

Now take $t\in \sfT^c$ and let $\sfP$ be a smashing prime with $t\notin \sfP$ i.e.\ $F_\sfP \otimes t \neq 0$ in $\sfT/\sfP$. Pick $\sfS$ such that $U_\sfS \cap V_\sfP = \{\sfP\}$. Because $\sfP$ is smashing $t\otimes F_\sfP$ is compact in $\sfT/\sfP$ and since $\sfP$ is prime any two non-zero smashing ideals of $\sfT/\sfP$ have non-zero intersection. The final preparatory observation we need is that the left idempotent for $\frac{\sfS \vee \sfP}{\sfP}$ in $\sfT/\sfP$ is $E_\sfS\otimes F_\sfP$. With this in mind we consider, in $\sfT/\sfP$,
\[
\loc^\otimes(t\otimes F_\sfP) \cap \frac{\sfS \vee \sfP}{\sfP} = \loc^\otimes(E_\sfS \otimes F_\sfP \otimes t)
\]
and see that it is zero if and only if $E_\sfS \otimes F_\sfP \otimes t$ is zero. We have assumed $t\otimes F_\sfP \neq 0$ and $\frac{\sfS \vee \sfP}{\sfP} \neq 0$ (since $\sfS \nsubseteq \sfP$) so this intersection cannot vanish. Hence $\sfP \in \ssupp (t)$ as desired.
\end{proof}

\begin{rem}
The topological condition $T_D$ is a familiar one, but in disguise. Indeed, (remembering that we are in the Hochster dual setting) a spectral space $A$ has the property that each of its points can be written as the intersection of a Thomason subset and the complement of a Thomason subset if and only if the Hochster dual $A^\vee$ is $T_D$. 

Note that it is possible to isolate the property of being $T_D$ using a more lattice theoretic formalism as we now describe. For a lattice $L$, we say that an element $b$ \emph{immediately precedes} $a$ if $b < a$ and if $b < c \leq a$ then we necessarily have $c=a$. A \emph{slicing filter} of a lattice $L$ is a prime filter $F$ such that there exist $a,b $ with $b \not\in F \ni a$, and $b$ immediately preceding $a$ in $L$. One can then show, as in~\cite{MR2868166}, that a $T_0$-space is $T_D$ if and only if each $N(x) = \{U \in \mcO(X) \mid x \in U\}$ is a slicing filter.
\end{rem}

Unlike its less refined cousin, the small smashing support does satisfy the half $\otimes$-formula.

\begin{prop}\label{prop:half}
Suppose that $\Spc^\mathrm{s}(\sfT)$ is $T_D$. If $t\in \sfT^c$ and $X\in \sfT$ then 
\[ \ssupp(t\otimes X) = \ssupp(t) \cap \ssupp(X).\]
\end{prop}
\begin{proof}
Let $t\in \sfT^c$ and $X\in \sfT$ and suppose that $\sfP \in \ssupp(t)$. In particular, $F_\sfP\otimes t$ is compact and not zero in $\sfT/\sfP$. As in the proof of Lemma~\ref{lem:bigvssmallsupp} we can consider $\sfS$ such that $U_\sfS \cap V_\sfP = \{\sfP\}$ and observe that in $\sfT/\sfP$,
\[
\loc^\otimes(F_\sfP\otimes t) \cap \frac{\sfS \vee \sfP}{\sfP} = \loc^\otimes(E_\sfS \otimes F_\sfP \otimes t) = \loc^\otimes(\Gamma_\sfP \otimes t) \neq 0.
\]
We claim that $\loc^\otimes(\Gamma_\sfP \otimes t) = \frac{\sfS \vee \sfP}{\sfP}$. We know $U_{\frac{\sfS \vee \sfP}{\sfP}} = \{0\}$ in $\Spc^\mathrm{s}(\sfT/\sfP)$ and so, by the above containment, $U_{\loc^\otimes(\Gamma_\sfP \otimes t)} = \{0\}$ and equality follows from spatiality of the frame of smashing ideals.

Suppose then that $\Gamma_\sfP \otimes t\otimes X \cong 0$. Then tensoring with $X$ kills $\loc^\otimes(\Gamma_\sfP \otimes t) = \frac{\sfS \vee \sfP}{\sfP}$ and so $\Gamma_\sfP \otimes X\cong 0$. 
\end{proof}

We conclude our musings on supports in the abstract by studying the small supports of left and right idempotents.

\begin{lem}\label{lem:ssuppE}
Let $\sfS$ be a smashing ideal with left and right idempotents $E_\sfS$ and $F_\sfS$. If $\sfP \in U_\sfS$ is locally closed then $\Gamma_\sfP \otimes E_\sfS \neq 0$ and $\Gamma_\sfP \otimes F_\sfS =0$.
\end{lem}
\begin{proof}
Let $\sfR$ be a smashing ideal such that $\{\sfP\} = U_\sfR \cap V_\sfP$. Our interest is in $E_\sfS \otimes \Gamma_\sfP = E_\sfS \otimes E_\sfR \otimes F_\sfP$. The idempotent $E_\sfS\otimes E_\sfR$ corresponds to $\sfS\cap \sfR$ i.e.\ to the open subset $U_\sfS\cap U_\sfR$. By assumption $\sfP$ lies in both of these opens and so
\[
\{\sfP\} = U_\sfR \cap V_\sfP = U_\sfR \cap U_\sfS \cap V_\sfP.
\]
By Lemma~\ref{lem:indep} we deduce that
\[
\Gamma_\sfP = E_\sfR \otimes F_\sfP \cong E_\sfR \otimes E_\sfS \otimes F_\sfP.
\]
It follows that $E_\sfS \otimes \Gamma_\sfP$ is non-zero and, in fact, that $\Gamma_\sfP$ lies in $\sfS$ so $\Gamma_\sfP \otimes F_\sfS =0$.
\end{proof}

\begin{prop}\label{prop:ssuppE}
Suppose that $\Spc^\mathrm{s}(\sfT)$ is $T_D$. Let $\sfS$ be a smashing ideal of $\sfT$. Then
\[
\ssupp (E_\sfS) = U_\sfS \text{ and } \ssupp (F_\sfS) = V_\sfS.
\]
\end{prop}
\begin{proof}
Let $\sfS$ be a smashing ideal. By Lemma~\ref{lem:ssuppE} we have $U_\sfS \subseteq \ssupp (E_\sfS)$. If $\sfP \in \ssupp (E_\sfS)$ then, in particular, $E_\sfS \otimes F_\sfP \neq 0$. Hence $E_\sfS \notin \sfP$, i.e.\ $\sfS \nsubseteq \sfP$ and so $\sfP \in U_\sfS$. This proves the claimed equality.

Similarly, we have shown in the last lemma that $\ssupp (F_\sfS) \subseteq V_\sfS$. If $\sfP \in V_\sfS$ then we have just shown that $E_\sfS \otimes \Gamma_\sfP = 0$. It follows that $F_\sfS \otimes \Gamma_\sfP \cong \Gamma_\sfP$ and so $\sfP \in \ssupp (F_\sfS)$.
\end{proof}

%

\subsection{Functoriality of the smashing spectrum}

We now discuss functoriality of the construction. Suppose that $F\colon \sfT \to \sfU$ is an exact and coproduct preserving symmetric monoidal functor between big tt-categories. That $\sfT$ and $\sfU$ are rigidly-compactly generated and $F$ is symmetric monoidal conspire to guarantee that $F$ preserves compact objects. 

\begin{lem}
If $\sfS$ is a smashing subcategory of $\sfT$ with left idempotent $E$ then $F(E)$ is a left idempotent in $\sfU$ and the associated smashing subcategory of $\sfU$ is $\loc^\otimes(F\sfS)$. The analogous statement holds for right idempotents. 
\end{lem}
\begin{proof}
Applying $F$ to $\varepsilon\colon E \to \unit_\sfT$ and using the inverse of $\unit_\sfU \stackrel{\sim}{\to} F(\unit_\sfT)$ gives the structure map for $F(E)$. One checks easily that this makes $F(E)$ into a left idempotent. The associated smashing subcategory is $\loc^\otimes(F(E))$, and it remains to note that if $X\in \loc^\otimes(E)$ then $F(X) \in \loc^\otimes(F(E))$ (e.g.\ apply \cite{StevensonActions}*{Lemma~3.8}).
\end{proof}

\begin{prop}
The functor $F$ induces a morphism of frames $f\colon \mcS(\sfT) \to \mcS(\sfU)$. In particular, there is a continuous map $\Spc^\mathrm{s}(\sfU) \to \Spc^\mathrm{s}(\sfT)$. 
\end{prop}
\begin{proof}
The induced morphism $f$ is defined by sending $\sfS$ to $\loc^\otimes(F\sfS)$ as in the previous lemma. If $\sfS_i$ is an $I$-indexed family of smashing ideals of $\sfT$ then
\begin{align}
f (\vee_i \sfS_i) &= f \loc^\otimes(\cup_i \sfS_i) \\
&= \loc^\otimes( F\loc^\otimes(\cup_i \sfS_i)) \\
&= \loc^\otimes (F(\cup_i \sfS_i)) \\
&= \loc^\otimes(\cup_i F\sfS_i) \\
&= \loc^\otimes(\cup_i \loc^\otimes(F\sfS_i)) \\
&= \vee_i f \sfS_i
\end{align}
where the 3rd and 5th equalities are the fact that applying the coproduct preserving exact monoidal functor $F$ commutes with forming localizing subcategories as shown in \cite{StevensonActions}*{Lemma~3.8}.

If $\sfS_i$ for $i=1,2$ are smashing subcategories with associated left idempotents $E_i$ then $\sfS_1 \cap \sfS_2$ corresponds to the left idempotent $E_1 \otimes E_2$. Thus, identifying the lattices of smashing ideals and left idempotents, we have
\[
f(E_1 \wedge E_2) = f(E_1 \otimes E_2) = F(E_1 \otimes E_2) = F(E_1)\otimes F(E_2) = F(E_1) \wedge F(E_2) = f(E_1) \wedge f(E_2)
\]
and so $f$ preserves finite meets. 

Thus $f$ is a map of frames. The existence of the continuous map on the associated spectra is automatic.
\end{proof}

\begin{rem}\label{rem:map}
Let $\sfP$ be a meet-prime smashing ideal of $\sfU$. There is a corresponding point $p$ of $\mcS(\sfU)$ which sends $\sfS$ to $0$ if and only if $\sfS\subseteq \sfP$. The map $\Spec(f)$ on spaces is defined at $p$ by pushing forward this point, i.e.\
\[
\Spec(f)(p) = p\circ f.
\]
We can formulate this solely in terms of meet-primes as
\[
\Spec(f)(\sfP) = \bigvee \{\sfS \in \mcS(\sfT) \mid F(\sfS)\subseteq \sfP\}.
\]
\end{rem}

\begin{rem}
If $F\colon \sfT \to \sfU$ is an exact and coproduct preserving symmetric monoidal functor between big tt-categories it is not necessarily true that $F^{-1}$ sends smashing ideals to smashing ideals, and so one cannot describe the induced map on smashing spectra in the same way as for the Balmer spectrum.
\end{rem}

We give a concrete example of this phenomenon.

\begin{ex}
Consider the derived base change functor $\pi\colon \sfD(\ZZ) \to \sfD(\FF_p)$ for a prime $p$. Then $0$ is the unique prime smashing ideal of $\sfD(\FF_p)$. We have
\[
\pi^{-1}(0) = \Ker \pi = \loc(\{\QQ\} \cup \{\FF_q \mid q\neq p\})
\]
which is not smashing (its support is not specialization closed). The induced map
\[
\Spec (\FF_p) = \Spc^\mathrm{s}(\sfD(\FF_p)) \to \Spc^\mathrm{s}(\sfD(\ZZ)) = \Spec (\ZZ)
\]
sends $0$ to $\sfD_{\mcZ(p)}(\ZZ)$ the smashing prime corresponding to $p$; see Section~\ref{sec:ringex} for further details.
\end{ex}

\section{The big spectrum}
\setcounter{subsection}{1}

In this section we will introduce yet another space associated to a big tt-category $\sfT$. Let us begin with a very naive definition which has, perhaps because of this extreme na\"ivet\'e, not yet been considered in the literature.

\begin{defn}\label{defn:bigprime}
We call a proper localizing ideal $\sfP$ of $\sfT$ \emph{prime} (or a \emph{big prime (ideal)} to emphasize the size) if it is radical, i.e.\ if $X^{\otimes n}\in \sfP$ then $X\in \sfP$, and whenever $\sfI \cap \sfJ \subseteq \sfP$ for radical localizing ideals $\sfI$ and $\sfJ$ we have $\sfI \subseteq \sfP$ or $\sfJ \subseteq \sfP$. In other words, $\sfP$ is meet-prime with respect to radical localizing ideals.
\end{defn}

\begin{notation}
For a localizing ideal $\sfI$ we denote its radical by $\sqrt{\sfI}$, which is the smallest radical localizing ideal containing it. For localizing ideals $\sfI$ and $\sfJ$ we denote by $\sfI\otimes \sfJ$ the smallest localizing subcategory containing all objects $I\otimes J$ for $I\in\sfI$ and $J\in \sfJ$.
\end{notation}

\begin{lem}\label{lem:cap}
For radical localizing ideals $\sfI$ and $\sfJ$ we have
\[
\sfI \cap \sfJ = \sqrt{\sfI \otimes \sfJ}.
\]

\end{lem}
\begin{proof}
Since $\sfI$ and $\sfJ$ are ideals, if $I\in \sfI$ and $J\in \sfJ$ then $I\otimes J \in \sfI\cap \sfJ$. Thus $\sfI\otimes \sfJ \subseteq \sfI \cap \sfJ$. Since $\sfI$ and $\sfJ$ are radical so is $\sfI\cap \sfJ$ and hence $\sqrt{\sfI\otimes \sfJ} \subseteq \sfI \cap \sfJ$. 

If $K \in \sfI \cap \sfJ$ then $K\otimes K \in \sfI \otimes \sfJ$ and so $K$ lies in $\sqrt{\sfI\otimes \sfJ}$, which proves the reverse containment and hence equality.
\end{proof}

\begin{lem}\label{lem:prime}
Suppose that $\sfI$ and $\sfJ$ are radical localizing ideals of $\sfT$ and $\sfP$ is a big prime. Then $\sfI \otimes \sfJ \subseteq \sfP$ implies that at least one of $\sfI$ or $\sfJ$ is contained in $\sfP$. 
\end{lem}
\begin{proof}
Since $\sfP$ is prime and hence radical by definition we have that $\sfI\otimes \sfJ \subseteq \sfP$ if and only if $\sqrt{\sfI\otimes \sfJ}$ is contained in $\sfP$. Thus if $\sfI\otimes \sfJ$ is contained in $\sfP$ so is its radical, which is $\sfI\cap \sfJ$ by Lemma~\ref{lem:cap}. The conclusion then follows from primeness of $\sfP$. 
\end{proof}

\begin{rem}
Ideally this would, as in commutative algebra and the tt-geometry of small tt-categories, be equivalent to the statement that $X\otimes Y\in \sfP$ if and only if at least one of $X$ or $Y$ lies in $\sfP$. The issue is that the presence of infinite coproducts makes taking the radical a much more complicated process (cf.\ the example that follows).
\end{rem}

\begin{ex}
Let $k$ be a field, let $n_i$ be a sequence of natural numbers with each $n_i \geq 2$, and consider the truncated polynomial ring 
\[
\Lambda = k[x_i \mid i\in \NN]/(x_i^{n_i} \mid i\in \NN).
\]
By \cite{DP08}*{Theorem~B} for any $n\geq 1$ there exists an object $X_n \in \sfD(\Lambda)$ such that $X_n^{\otimes n}\not\cong 0$ but $X_n^{\otimes n+1} \cong 0$. It follows that
\[
\sfN = \{X\in \sfD(\Lambda) \mid X^{\otimes i} \cong 0 \text{ for some } i\in \NN\}
\]
is not a localizing ideal. It is closed under suspensions, summands, and cones but not coproducts: the object $\coprod_n X_n$ is not tensor nilpotent and so does not lie in $\sfN$.

One can of course close under coproducts, adding such sums of nilpotents, but then there is no obvious reason for the result to be closed under cones. Once we add cones there is no obvious reason for the result to be radical. We continue in this fashion and the process does not obviously ever terminate (even proceeding transfinitely and taking unions at limit ordinals, as these unions might not be closed under coproducts).
\end{ex}

\begin{lem}\label{lem:radical}
Suppose that every localizing ideal of $\sfT$ is radical. Then $\sfP$ is a big prime if and only if whenever $X\otimes Y \in \sfP$ one of $X$ or $Y$ lies in $\sfP$.
\end{lem}
\begin{proof}
Suppose that $\sfP$ is a big prime and $X\otimes Y \in \sfP$. Then, using Lemma~\ref{lem:cap} and the fact that radicals are free, we see
\[
\loc^\otimes(X) \cap \loc^\otimes(Y) = \loc^\otimes(X) \otimes \loc^\otimes(Y) = \loc^\otimes(X\otimes Y) \subseteq \sfP
\]
(the second equality via \cite{StevensonActions}*{Lemma~3.11}) and so $X\in \sfP$ or $Y\in \sfP$. 

On the other hand, suppose that $\sfP$ satisfies the objectwise condition. It is automatically radical by assumption. Let us be given $\sfI \cap \sfJ \subseteq \sfP$ for localizing ideals $\sfI$ and $\sfJ$ and say $\sfI \nsubseteq \sfP$. Pick an $X\in \sfI$ with $X\notin \sfP$. Then for $Y\in \sfJ$ we have
\[
X\otimes Y \in \sfI \cap \sfJ \subseteq \sfP
\]
and so, since $X\notin \sfP$, we have $Y\in \sfP$. This shows $\sfJ \subseteq \sfP$ as required.
\end{proof}

\begin{rem}
The assumption that every localizing ideal is radical is quite strong. For instance, it rules out the presence of nilpotent objects and these exist in cases of great interest e.g.\ the stable homotopy category. However, there are examples: if $\sfT$ is a big tt-category such that the assignment sending a subset of the spectrum (either $\Spc^\mathrm{s} (\sfT)$ or $\Spc(\sfT^c)$) to the ideal of objects supported on that subset is a bijection then every localizing ideal is radical (simply because $\Supp X^{\otimes n} \subseteq \Supp X$). This is known for several classes of examples, the prototypical one being $\sfD(R)$ for a commutative noetherian ring $R$ \cite{NeeChro}.
\end{rem}

\begin{defn}\label{defn:bigspc}
We denote by $\SPC (\sfT)$ the collection of big prime ideals of $\sfT$, and call it the \emph{big spectrum}. We define the \emph{big support} of an object $X\in \sfT$ by
\[
\SUPP (X) = \{\sfP \in \SPC (\sfT) \mid X\notin \sfP\}
\]
and `topologize' $\SPC (\sfT)$ (see Remark \ref{rem:smaug}) by declaring these subsets to be closed and taking the topology they generate.
\end{defn}

\begin{rem}\label{rem:smaug}
There be dragons here. We do not know that $\SPC (\sfT)$ forms a set in general, and so it is potentially an abuse of the usual terminology to call it a space.
\end{rem}

\begin{rem}\label{rem:topologies}
One could equally well topologize $\SPC (\sfT)$ by only looking at the big support of $\alpha$-compact objects for some regular cardinal $\alpha$. This gives a family of topologies, which become increasingly fine, that seem worthy of investigation. A particular case of interest, which will occur in the sequel, is to take $\alpha = \aleph_0$.
\end{rem}

The big support satisfies the properties one would expect and, at least when one doesn't need to deal with radicals, we also get the tensor product property.

\begin{lem}
The big support satisfies the following properties:
\begin{itemize}
\item[(1)] $\SUPP(0) = \varnothing$ and $\SUPP(\unit) = \SPC (\sfT)$;
\item[(2)] $\SUPP (\coprod_i X_i) = \cup_i \SUPP (X_i)$;
\item[(3)] $\SUPP (\Sigma X) = \SUPP (X)$;
\item[(4)] for any triangle $X \to Y \to Z$ we have $\SUPP (Y) \subseteq \SUPP (X) \cup \SUPP (Z)$;
\item[(5)] $\SUPP (X\otimes Y) \subseteq \SUPP (X)\cap \SUPP (Y)$;
\item[(6)] if every localizing ideal is radical then $\SUPP (X \otimes Y) = \SUPP (X)\cap \SUPP (Y)$.
\end{itemize}
\end{lem}
\begin{proof}
The usual results with the usual proof; property (6) follows immediately from Lemma~\ref{lem:radical}.
\end{proof}



\section{Comparison maps}\label{sec:comparison}

There is some work to do in order to situate $\Spc^\mathrm{s}(\sfT)$ and $\SPC (\sfT)$ in the landscape: we need to describe their relation to $\Spc (\sfT^c)$, the premier, and to the plucky newcomer $\Spc^\mathrm{h}(\sfT^c)$, the homological spectrum of \cite{BaHomological}, both of which we recalled in Section~\ref{sec:prelim}. 

Our interest in our new toys is the hope that they can be used to classify (more of the) localizing tensor ideals in situations where the telescope conjecture fails or where $\sfT$ is not `sufficiently noetherian' (whatever that might mean). We already know that $\Spc$ and $\Spc^\mathrm{h}$ are not up to this task. Nonetheless there are canonical morphisms comparing these spaces and it is instructive, and computationally beneficial, to explore them.

\subsection{The smashing spectrum}

Suppose that $\sfK^c$ is a thick ideal of $\sfT^c$. Then we may inflate it to $\sfK = \loc(\sfK^c)$. By results of Miller and Neeman the localizing subcategory $\sfK$ is a tensor ideal and satisfies $\sfK \cap \sfT^c = \sfK^c$, justifying the notation (a convenient reference for these facts is \cite{BaRickard}*{Theorem~4.1}). On the other hand, given a smashing ideal $\sfS$ of $\sfT$ we may form $\sfS^c = \sfS\cap \sfT^c$ and this is a tensor ideal of $\sfT^c$.

Let us denote the lattice of thick tensor ideals of $\sfT^c$ by $\mcT(\sfT^c)$, and we remind the reader that $\mcS(\sfT)$ is the lattice of smashing tensor ideals. Inflation gives a map of posets $f\colon \mcT(\sfT^c) \to \mcS(\sfT)$ and intersection with compacts gives a poset map $g\colon \mcS(\sfT) \to \mcT(\sfT^c)$.

\begin{lem}\label{lem:inflation}
The poset map $f$ is a morphism of frames with right adjoint $g$ and $gf = 1$.
\end{lem}
\begin{proof}
By construction $f$ preserves arbitrary joins. Suppose that $\sfK_i^c$ for $i=1,2$ are thick ideals of compacts. Then $f(\sfK_1^c \cap \sfK_2^c)$ is clearly contained in $f(\sfK_1^c) \cap f(\sfK_2^c)$. For the reverse containment we note that, letting $E_i$ denote the respective left idempotents, $E_i \in \loc(\sfK_i^c)$ implies that $E_1\otimes E_2$, the idempotent for $f(\sfK_1^c) \cap f(\sfK_2^c)$, lies in
\[
\loc(\sfK_1^c) \otimes \loc(\sfK_2^c) = \loc(\sfK_1^c \otimes \sfK_2^c)
\]
where this equality is \cite{StevensonActions}*{Lemma~3.11}. Thus $E_1\otimes E_2$, and hence $f(\sfK_1^c) \cap f(\sfK_2^c)$, is contained in $f(\sfK_1^c \cap \sfK_2^c)$.

We have noted above that, by work of Neeman, $\loc(\sfK^c) \cap \sfT^c = \sfK^c$ i.e.\ $gf=1$. Finally, suppose that $\sfK^c \in \mcT(\sfT^c)$ and $\sfS \in \mcS(\sfT)$. Then $f(\sfK^c) \subseteq \sfS$ if and only if $\sfK^c \subseteq \sfS$ if and only if $\sfK^c \subseteq \sfS^c = g(\sfS)$, which proves that $f\dashv g$.
\end{proof}

Recall that for a spectral space $X$ we denote by $X^\vee$ the Hochster dual of $X$. This is the space with the same points and whose open subsets are generated by the closed subsets of $X$ with quasi-compact open complement.

\begin{thm}\label{thm:comparison}
Suppose that $\mcS(\sfT)$ is spatial. There is a continuous map $\psi\colon \Spc^\mathrm{s}(\sfT) \to (\Spc (\sfT^c))^\vee$ given by $\psi(\sfP) = \sfP^c$ such that 
\[
\psi^{-1}\Supp (t) = \sSupp (t)
\]
for all $t\in \sfT^c$.
\end{thm}
\begin{proof}
The map $\psi$ is given by $\Spec(f)$, i.e.\ by Stone duality, and is automatically continuous. An exercise in unwinding this construction shows that it sends a smashing prime $\sfP$ to $\psi(\sfP) = g(\sfP) = \sfP^c$. 

Suppose that $t$ is compact. By definition, $\sfP \in \sSupp (t)$ means that $t\notin \sfP$, i.e.\ $t\notin \sfP^c$. This says precisely that $t\notin \psi(\sfP)$. We have shown that
\[
\psi^{-1}\Supp (t) = \{\sfP \mid t\notin \psi(\sfP)\} = \{\sfP \mid t\notin \sfP^c\} = \{\sfP \mid t\notin \sfP\} = \sSupp (t)
\]
as claimed.
\end{proof}

\begin{rem}
It is not in general true that $\psi$ is compatible with the small support of arbitrary objects, see Remark~\ref{rem:ssupppreimage}.
\end{rem}

\begin{rem}
Even if $\mcS(\sfT)$ were not spatial one can perform the same construction to produce a map as in the theorem.
\end{rem}

\begin{cor}\label{cor:telescope}
Suppose that $\mcS(\sfT)$ is spatial. The comparison map $\psi$ is a homeomorphism if and only if the telescope conjecture for smashing ideals holds, i.e.\ if every smashing ideal is generated by compact objects. In this case the Balmer--Favi and small smashing supports coincide (at points where they are defined).
\end{cor}
\begin{proof}
As Stone duality is an equivalence of categories the map $\psi$ is a homeomorphism if and only if the maps $f$ and $g$ of Lemma~\ref{lem:inflation} are inverse lattice isomorphisms between $\mcS(\sfT)$ and $\mcT(\sfT^c)$. This is the case precisely if the telescope conjecture holds.

In this situation it follows that, when they can be defined, the idempotents yielding the two notions of support coincide. For convenience of notation let us treat $\psi$ as an identification. Given $\mfp \in \Spc (\sfT^c)$ its closure $\mcV(\mfp)$ is open in $\Spc^\mathrm{s}(\sfT)$ and $\sfP$, the corresponding smashing prime, is maximal in it (remember if $\mfp \subseteq \mfq$ then $\Supp(\mfp) \supseteq \Supp(\mfq)$). Thus we may compute $\Gamma_\sfP$ as 
\[
\Gamma_\sfP = E_{\sfT_{\mcV(\mfp)}} \otimes F_\sfP = \Gamma_{\mcV(\mfp)}\unit \otimes L_{\Supp(\mfp)}\unit = \Gamma_\mfp \unit.
\]
\end{proof}

\begin{rem}\label{rem:fail}
We learn from the Corollary that there are cases in which the smashing spectrum is insufficient for the purposes of understanding localizing ideals. For instance, it is shown in \cite{BazzoniStovicek} that the telescope conjecture holds for $R = \prod_\NN k$, where $k$ is a field, and so $\psi$ is a homeomorphism. But, as shown in \cite{StevensonAbsFlat} there are exotic localizing ideals, i.e.\ ideals which are not determined by their support.
\end{rem}

\begin{rem}
Of course if the telescope conjecture holds then $\mcS(\sfT)$ is spatial. On the other hand, one has to consider the possibility that $\mcS(\sfT)$ is not spatial, and the telescope conjecture fails, but the comparison map is a homeomorphism. Such an example would, needless to say, be very informative.
\end{rem}

%

%
%

\subsection{Primes, primes, and primes}

Let us now discuss how to compare the remaining notions of prime ideal, i.e.\ smashing, big, and homological.

Given $\mcB \in \Spc^\mathrm{h}(\sfT^c)$, with associated pure injective $I_\mcB$ (we change the usual notation slightly to avoid conflict with left idempotents), Balmer tells us in \cite{BaHomological}*{Theorem~3.11} that
\[
\mcB' = \Ker [-, I_\mcB] = \{M\in \Modu \sfT^c \mid [M,I_\mcB] = 0\},
\]
where $[-,-]$ denotes the internal hom in $\Modu \sfT^c$, is the unique maximal localizing Serre ideal containing $\mcB$.

\begin{lem}\label{lem:homologicaltobig}
Suppose that $\mcB \in \Spc^\mathrm{h}(\sfT^c)$ and let $\mcB'$ be the unique maximal localizing Serre ideal containing $\mcB$. Then $h^{-1}\mcB'$ is a big prime in $\sfT$.
\end{lem}
\begin{proof}
First of all, since $\mcB'$ is a localizing subcategory of $\Modu \sfT^c$ closed under the induced suspension, it follows from the fact that $h$ is cohomological, coproduct preserving, and compatible with suspension that $h^{-1}\mcB'$ is localizing in $\sfT$. Similarly, since $h$ is monoidal it follows that $h^{-1}\mcB'$ is an ideal.

Next let us check that $\sfP = h^{-1}\mcB'$ is radical. Suppose $X^{\otimes n} \in \sfP$, i.e.\ $hX^{\otimes n} \in \mcB'$. This is the same as saying $\mcB \notin \hsupp (X^{\otimes n})$. By the tensor product formula \cite{BaHomological}*{Theorem~4.5} the homological supports of $X^{\otimes n}$ and $X$ agree, so $\mcB$ is not in the support of $X$, i.e.\ $hX\in \mcB'$ and so $X\in \sfP$.

Finally, we must check $\sfP$ satisfies the primeness condition. To this end let $\sfI$ and $\sfJ$ be radical localizing ideals with $\sfI \otimes \sfJ \subseteq \sfP$ (we can work with this condition by Lemma~\ref{lem:cap}). Suppose that $\sfI \nsubseteq \sfP$ and pick some $X$ witnessing this fact. Thus $hX\notin \mcB'$ and so $\overline{h}X$, its image in $\Modu \sfT^c/\mcB'$ is also non-zero. The object $\overline{h}X$ is flat and so $\Ker(-\otimes \overline{h}X)$ is a proper localizing Serre ideal of $\Modu \sfT^c/\mcB'$. By maximality of $\mcB'$ it must thus be trivial, i.e.\ $\overline{h}X \otimes(-)$ kills no non-zero object. But for every $Y\in \sfJ$ we know that $hX \otimes hY \cong h(X\otimes Y) \in \mcB'$ and so $\overline{h}X \otimes \overline{h}Y \cong 0$ in the quotient. Thus $\overline{h}Y \cong 0$ i.e.\ $hY \in \mcB'$ and so $Y\in \sfP$. Hence $\sfP$ is a big prime.
\end{proof}

\begin{prop}\label{prop:homologicaltobig}
The assignment $\mcB \mapsto \mcB' \mapsto h^{-1}\mcB'$ defines a comparison map
\[
\chi \colon \Spc^\mathrm{h}(\sfT^c) \to \SPC (\sfT) \text{ such that } \chi^{-1}\SUPP (X) = \hsupp (X)
\]
for every object $X$ of $\sfT$. In particular, $\chi$ is continuous when the big spectrum is topologized via the supports of compact objects.
\end{prop}
\begin{proof}
The map $\chi$ is well defined by Lemma~\ref{lem:homologicaltobig}. 
For an object $X$ we have
\begin{align*}
\chi^{-1} \SUPP (X) &= \chi^{-1} \{\sfP \in \SPC (\sfT) \mid X\notin \sfP\} \\
&= \{\mcB \mid X\notin h^{-1}\mcB'\} \\
&= \{\mcB \mid h(X) \notin \mcB'\} \\
&= \{\mcB \mid [X, I_\mcB] \neq 0\} \\
&= \hsupp(X)
\end{align*}
where the last equalities are from \cite{BaHomological}*{Definition~4.1 and Proposition~4.2}. It follows that $\chi$ is continuous as the $\SUPP (x)$ are the basic closed subsets for the chosen topology on the big spectrum.
\end{proof}

\begin{rem}\label{rem:speculation1}
We lack the examples to predict how this map might behave in general. We would hazard the (very speculative and particularly baseless) guess that there should be a comparison map in the other direction (of which $\chi$ would like to be a section\textemdash{}see Proposition~\ref{prop:bigtosmall} for a vague step in this direction).
\end{rem}

What we have shown is also enough to guarantee at least as many big primes as small ones.

\begin{thm}\label{thm:enoughbigprimes}
Given $\mathbf{p} \in \Spc (\sfT^c)$ there exists a big prime $\sfP$ such that $\mathbf{p} = \sfP^c$. 
\end{thm}
\begin{proof}
By \cite{BaNilpotence}*{Corollary~3.9} there is a homological prime $\mcB$ in $\modu \sfT^c$ such that $\mathbf{p} = h^{-1}\mcB$. Passing to the unique localizing Serre ideal $\mcB'$ containing $\mcB$ we get a big prime $\sfP = h^{-1}\mcB'$ by Lemma~\ref{lem:homologicaltobig}. By \cite{BaHomological}*{Theorem~3.11} $\mcB' \cap \modu \sfT^c = \mcB$ and so
\[
\sfP^c = h^{-1}\mcB' \cap \sfT^c = h^{-1}\mcB' \cap h^{-1} (\modu \sfT^c) = h^{-1}(\mcB' \cap \modu \sfT^c) = h^{-1}\mcB = \mathbf{p}.
\]
\end{proof}

\begin{prop}\label{prop:bigtosmall}
Suppose that every localizing ideal of $\sfT$ is radical and let $\sfP \in \SPC (\sfT)$. Then $\sfP^c = \sfP \cap \sfT^c$ lies in $\Spc (\sfT^c)$. In particular there is a comparison map (of sets) $\omega\colon \SPC (\sfT) \to \Spc (\sfT^c)$ fitting into a commutative triangle
\[
\xymatrix{
\Spc^\mathrm{h}(\sfT^c) \ar[rr]^-\phi \ar[dr]_-\chi && \Spc (\sfT^c) \\
& \SPC (\sfT) \ar[ur]_-\omega &
}
\]
\end{prop}
\begin{proof}
By Lemma~\ref{lem:radical} a localizing ideal $\sfP$ is a big prime if and only if $X\otimes Y \in \sfP$ implies at least one of $X$ or $Y$ lies in $\sfP$. Thus if $a$ and $b$ are compact objects of $\sfT$ such that $a\otimes b\in \sfP^c$ we must have, without loss of generality, $a\in \sfP$. Since $a$ is compact it lies in $\sfP^c$ and so $\sfP^c$ is a prime thick tensor ideal in the sense of Balmer. 

To see the triangle commutes we just compute that, for a homological prime $\mcB$ contained in its unique maximal localizing Serre ideal $\mcB'$, we have
\[
\phi(\mcB) = h^{-1}\mcB = h^{-1}\mcB' \cap \sfT^c = \omega(h^{-1}\mcB') = \omega\chi \mcB
\]
analogously to the computation in Theorem~\ref{thm:enoughbigprimes}.
\end{proof}



\section{An example: commutative noetherian rings}\label{sec:ringex}
\setcounter{subsection}{1}

In this section we discuss, in an example we understand somewhat well, the behaviour of the objects we have defined.

\begin{rem}
We stick here to the derived category of a ring, however many parts easily generalize to any $\sfT$ when we know the classification of localizing ideals in terms of $\Spc(\sfT^c)$. By results of \cite{BarthelStrat}, we have such a classification of localizing tensor ideals when $\sfT$ satisfies the local-to-global principle, and each $\Gamma_\sfP \sfT$ is a minimal localizing ideal. 
\end{rem}


Throughout we fix $R$ a commutative noetherian ring. All functors are taken to be derived, and equality sometimes means isomorphism. Let us begin by reminding ourselves of some details to fix notation.

Let $\mfp$ be a prime ideal of $R$. We set, as usual, $\mcV(\mfp) = \overline{\{\mfp\}}$ and 
\begin{displaymath}
\mcZ(\mfp) = \{\mfq \in \Spec (R) \mid \mfp \notin \mcV(\mfq)\} = \Spec(R) - \Spec(R_\mfp).
\end{displaymath}
There is a corresponding prime ideal in $\sfD^\mathrm{perf}(R)$, namely
\begin{displaymath}
\sfD^\mathrm{perf}_{\mcZ(\mfp)}(R) = \{E\in \sfD^\mathrm{perf}(R) \mid \supp (E) \subseteq \mcZ(\mfp)\} = \{E\in \sfD^\mathrm{perf}(R)\mid E_\mfp = 0\}.
\end{displaymath}
This gives an inclusion reversing bijection between $\Spec (R)$ and $\Spc (\sfD^\mathrm{perf}(R))$.

We know, by the work of Neeman \cite{NeeChro}, that localizing ideals of $\sfD(R)$ are in bijection with subsets of $\Spec (R)$ and that the telescope conjecture holds.

By \cite{BaNilpotence}*{Corollary~5.11} the comparison map $\Spc^\mathrm{h}(\sfD^\mathrm{perf}(R)) \to \Spc (\sfD^\mathrm{perf}(R))$ is a bijection, i.e.\ there is a unique homological prime associated to each prime ideal of $R$. The homological prime associated to $\mfp$ is given by the kernel of the functor
\[
\modu \sfD^\mathrm{perf}(R) \to \modu \sfD^\mathrm{perf}(k(\mfp))
\]
induced by $R\to k(\mfp)$. 

Now let us describe the big and smashing primes for $\sfD(R)$ and then make a systematic comparison.


\subsection{Big primes}\label{ssec:bp}

We know that there are only a set of localizing subcategories in $\sfD(R)$, that every localizing subcategory is an ideal, and that every ideal is radical. So we can comfortably look at prime localizing ideals and by Lemma~\ref{lem:radical} we can define these as either meet-prime localizing subcategories or as objectwise prime with respect to the tensor product.

\begin{lem}
Let $\sfP$ be a prime localizing ideal of $\sfD(R)$. Then there is a unique $\mfp\in \Spec (R)$ such that $k(\mfp)\notin \sfP$, and so it follows from the classification that
\begin{displaymath}
\sfP = \loc(k(\mfq) \mid \mfq \neq \mfp).
\end{displaymath}
In particular, the primes are the maximal proper localizing subcategories. This sets up a bijection $\Spec (R) \cong \SPC (\sfD(R))$.
\end{lem}
\begin{proof}
If $\mfp \neq \mfq$ then $k(\mfp)\otimes k(\mfq) = 0$ and so at least one of $k(\mfp)$ or $k(\mfq)$ lies in $\sfP$. In other words there is at most one $\mfp$ such that $k(\mfp)\notin \sfP$. We know the residue fields generate $\sfD(R)$ so as $\sfP$ is proper there is exactly one such $\mfp$. The remaining statements follow: the first by the classification of localizing subcategories and the second immediately.
\end{proof}

The corresponding notion of support is
\begin{displaymath}
\SUPP (X) = \{\sfP\in \SPC (\sfD(R)) \mid X\notin \sfP\} = \{\mfp \in \Spec (R) \mid k(\mfp)\otimes X\neq 0\},
\end{displaymath}
which recovers Foxby's small support, aka the Balmer-Favi notion of support under the identification $\Spc (\sfD^\mathrm{perf}(R)) \cong \SPC (\sfD(R))$. There are two extremal ways to topologise $\SPC (\sfD(R))$. One is to take as a basis of closed subsets the supports of the compact objects and this yields a homeomorphism to $\Spc (\sfD^\mathrm{perf}(R))$ and $\Spec (R)$. The other is to take as a basis the supports of arbitrary objects of $\sfD(R)$ which gives the discrete topology on $\SPC (\sfD(R))$ and a bijection between closed subsets and localizing ideals. See Section~\ref{ssec:top} for some remarks on what happens in between.

\begin{lem}
For a prime $\mfp\in \Spec (R)$ with corresponding small prime $\mathbf{p}$ and big prime $\sfP$ we have
\begin{displaymath}
\sfP \cap \sfD^\mathrm{perf}(R) = \mathbf{p} \text{ and hence } \loc(\mathbf{p}) \subseteq \sfP
\end{displaymath}
with the latter being an equality if and only if $\mfp$ is a generic point.
\end{lem}
\begin{proof}
We compute that
\begin{align*}
\sfP \cap \sfD^\mathrm{perf}(R) &= \{E\in \sfD^\mathrm{perf}(R) \mid \supp (E) \subseteq \Spec (R)\setminus \{\mfp\}\} \\
&= \{E\in \sfD^\mathrm{perf}(R) \mid \supp (E) \subseteq \mcZ(\mfp)\} \\
&= \sfD^\mathrm{perf}_{\mcZ(\mfp)}(R) \\
&= \mathbf{p}
\end{align*}
where the second equality is given by the fact that the support of a compact is closed, and so if it fails to contain $\mfp$ it cannot contain any prime specializing to $\mfp$.

We see that $\loc(\mathbf{p}) \subseteq \sfP$. This is an equality precisely if $\mcZ(\mfp) = \Spec (R)\setminus \{\mfp\}$, i.e.\ if and only if $\mfp$ is the unique prime specializing to $\mfp$; this is precisely the statement that $\mfp$ is generic.
\end{proof}


\subsection{Smashing primes}\label{ssec:sp}

Now let us turn to the characterization of smashing primes. The telescope conjecture holds and so by Corollary~\ref{cor:telescope} the comparison map
\[
\psi\colon \Spc^\mathrm{s}(\sfD(R)) \to (\Spc (\sfD^\mathrm{perf}(R)))^\vee
\]
is a homeomorphism and it identifies the smashing and usual support for compact objects. From this bijection (and the further identification of $\Spc (\sfD^\mathrm{perf}(R))$ with $\Spec (R)$) we see that the meet-prime smashing ideals are precisely the $\sfD_{\mcZ(\mfp)}(R)$ for $\mfp\in \Spec (R)$. 

Now let us discuss the support of a general object. It turns out that the smashing support is not granular enough to deal with arbitrary objects.

\begin{lem}\label{lem:exsupp}
Let $X$ be an object of $\sfD(R)$. Then $\sSupp (X)$ is the specialization closure of $\supp (X)$.
\end{lem}
\begin{proof}
Using the identification of the smashing spectrum with $\Spec (R)$ we compute that
\begin{align*}
\sSupp (X) &= \{\mfp \in \Spec (R) \mid X\notin \sfD_{\mcZ(\mfp)}(R)\} \\
&= \{\mfp \in \Spec (R) \mid \supp (X)\nsubseteq \mcZ(\mfp)\}
\end{align*}
which is precisely the set of specializations of points in $\supp (X)$.
\end{proof}

However, since $\Spec (R)$ is noetherian the dual space $(\Spec (R))^\vee$ is $T_D$ so we can work with the small smashing support of Definition~\ref{defn:smallsupport} and it does the job.

\begin{lem}\label{lem:reverseit}
Treating the bijection $\psi$ as an identification we have, for $X\in \sfD(R)$, an equality
\[
\ssupp (X) = \supp (X).
\]
In particular, the small smashing support classifies localizing subcategories of $\sfD(R)$.
\end{lem}
\begin{proof}
This is a special case of the final statement of Corollary~\ref{cor:telescope}.
\end{proof}


\subsection{Comparisons}

Let us now summarise how to move between these four equivalent setups. We already have comparison maps
\[
\phi \colon \Spc^\mathrm{h}(\sfD^\mathrm{perf}(R)) \to \Spc (\sfD^\mathrm{perf}(R)) \text{ and } \psi\colon \Spc^\mathrm{s}(\sfD(R)) \to (\Spc (\sfD^\mathrm{perf}(R)))^\vee
\]
which are given by $\phi(\mcB) = h^{-1}\mcB \cap \sfD^\mathrm{perf}(R)$ and $\psi(\sfP) = \sfP \cap \sfD^\mathrm{perf}(R)$. We know that these are both homeomorphisms. We also have a map
\[
\SPC(\sfD(R)) \to \Spc (\sfD^\mathrm{perf}(R)) \quad \sfQ \mapsto \sfQ \cap \sfD^\mathrm{perf}(R)
\]
which is continuous in either topology on the big spectrum, and a homeomorphism if we topologize it via supports of perfect complexes.

We also understand very explicitly how to move from homological primes to big primes. By Proposition~\ref{prop:homologicaltobig} there is a comparison map $\chi\colon \Spc^\mathrm{h}(\sfD^\mathrm{perf}(R)) \to \SPC (\sfD(R))$.

\begin{lem}\label{lem:exchi}
The comparison map $\chi$ is a homeomorphism for the topology on the big spectrum generated by supports of compacts. The homological prime $\mcB$ corresponding to $k(\mfp)$ is sent to 
\[
\chi(\mcB) = \Ker(k(\mfp)\otimes -) = \loc(k(\mfq) \mid \mfq \neq \mfp).
\]
\end{lem}
\begin{proof}
This is an exercise in unwinding the definitions and applying \cite{BaNilpotence}*{Corollary~5.11} to relate the homological primes to the kernels of the base change functors associated to the residue fields.
\end{proof}

We see that all four spaces are homeomorphic and all notions of support coincide (apart from the usual big/small distinction). This is indicative of the situations in which we have a classification, i.e.\ where all things are as tame as one could hope for.

\subsection{Topologies on the big spectrum}\label{ssec:top}

We now briefly discuss the content of Remark~\ref{rem:topologies}. For the sake of convenience we will frequently identify $\SPC (\sfD(R))$ and $\Spec (R)$ as discussed in Section~\ref{ssec:bp}. Given a regular cardinal $\alpha$ let us denote by $\tau_\alpha$ the topology on $\SPC (\sfD(R))$ obtained by generating the closed subsets via the collection
\[
\{\SUPP (X) \mid X\in \sfD(R)^\alpha\}
\]
where $\sfD(R)^\alpha$ denotes the full subcategory of $\alpha$-compact objects, and by $\tau_\infty$ the topology generated by taking supports of arbitrary objects.

We have already seen that $\tau_{\aleph_0}$ gives the usual topology on $\SPC (\sfD(R)) \cong \Spec (R)$. At the other extreme, the topology $\tau_\infty$ is discrete: for any subset $W$ we have
\[
\SUPP \big( \bigoplus_{\mfp \in W} k(\mfp) \big)= W.
\]
We note that in $\tau_\alpha$ any union of fewer than $\alpha$ closed subsets is still closed\textemdash{}this topology is what one might call $\alpha$-Alexandrov.

As a first approximation one can understand this metamorphosis from the Zariski to the discrete topology in terms of the cardinality of $R$.

\begin{lem}\label{lem:card1}
If $\mcV$ is a closed subset of $\Spec (R)$ then the corresponding left idempotent $\Gamma_\mcV R$ is $\aleph_1$-compact.
\end{lem}
\begin{proof}
For $r\in R$ the stable Koszul complex $K_\infty(r) = (R \to R_r)$, where the map is the canonical one and the complex lives in degrees $0$ and $1$, gives an explicit representative for $\Gamma_{\mcV(r)}R$. We have
\[
K_\infty(r) = \hocolim_i K(r^i)
\]
where $K(r^i)$ is the usual Koszul complex on $r^i$, and so we see $K_\infty(r)$ is a countable homotopy colimit of compacts and hence is $\aleph_1$-compact.

For the general case one can write $\mcV = \mcV(I)$ and, choosing generators $I = (r_1,\ldots,r_n)$, we have 
\[
\Gamma_\mcV R \cong K_\infty(r_1,\ldots,r_n) = K_\infty(r_1) \otimes \cdots \otimes K_\infty(r_n).
\]
One sees easily from what we have already shown that this is $\aleph_1$-compact, for instance by using that $\sfD(R)^{\aleph_1}$ is closed under countable homotopy colimits.
\end{proof}

For a regular cardinal $\alpha$ we denote by $\alpha^+$ its successor.

\begin{lem}\label{lem:card2}
Let $\mfp\in \Spec R$ be a prime ideal such that $S = R\setminus \mfp$ has cardinality $\alpha$. Then $R_\mfp$ is $\alpha^+$-presentable as a module and $\alpha^+$-compact in $\sfD(R)$.
\end{lem}
\begin{proof}
We can write $R_\mfp$ as a filtered colimit of copies of $R$ with index set $S$. Thus $R_\mfp$ is $\alpha^+$-presentable and so, for instance by appealing to \cite{KrAuslander}*{Theorem~5.10}, is $\alpha^+$-compact in $\sfD(R)$.
\end{proof}

\begin{rem}
We note that, quite frequently, this bound is far from optimal.
\end{rem}

\begin{prop}
Let $\mfp\in \Spec (R)$ be a prime ideal such that $S = R\setminus \mfp$ has cardinality $\alpha\geq \aleph_0$. Then $\Gamma_\mfp R$ is $\alpha^+$-compact, and so $\{\mfp\}$ is closed in $\tau_{\alpha^+}$. In particular, if $\alpha \geq \max\{\vert R \vert, \vert \Spec R \vert\}$ then $\tau_{\alpha^+}$ is the discrete topology on $\Spec (R)$.
\end{prop}
\begin{proof}
By Lemma~\ref{lem:card2} $R_\mfp$ is $\alpha^+$-compact and by Lemma~\ref{lem:card1} so is $\Gamma_{\mcV(\mfp)}R$. Hence $\Gamma_\mfp R = R_\mfp \otimes \Gamma_{\mcV(\mfp)}R$ is also $\alpha^+$-compact. It has support precisely $\{\mfp\}$ and so this subset is closed in $\tau_{\alpha^+}$. The second statement follows as, by assumption on $\alpha$, any set $W$ of primes is $\alpha^+$-small and each $\Gamma_\mfp R$ is $\alpha^+$-compact so
\[
W = \SUPP \big( \bigoplus_{\mfp \in W} \Gamma_\mfp R \big)
\]
is closed in $\tau_{\alpha^+}$.
\end{proof}



\section{An example: no telescope here}\label{sec:notelescope}
\setcounter{subsection}{1}


Let $(A,\mfm,k)$ be a non-noetherian rank $1$ valuation domain, e.g.\ we could take for $A$ the perfection of $\FF_p[[x]]$. 
The corresponding scheme $\Spec (A)$ has underlying space 

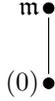
\begin{displaymath}
	\begin{tikzpicture}
    
		\node (v0) at (0,0) {};
    \node (va) at (0,1) {};

    \draw[fill] (v0)  circle (2pt) node [left] {$(0)$};
		\draw[fill] (va)  circle (2pt) node [left] {$\mfm$};

		\path[-] (v0) edge  node [above] {} (va);

  \end{tikzpicture}
\end{displaymath}
as for a discrete valuation ring. The maximal ideal $\mfm$ is flat with $\mfm^2 = \mfm$ and so $\mfm \to A$ is a left idempotent in $\sfD(A)$. This gives rise to a non-finite smashing localization
\begin{displaymath}
\loc(\mfm) \to \sfD(A) \to \sfD(k).
\end{displaymath}


\subsection{The smashing spectrum}


All smashing subcategories of $\sfD(A)$ are known by \cite{BazzoniStovicek} and fit into the following picture, where $Q$ denotes the function field of $A$:

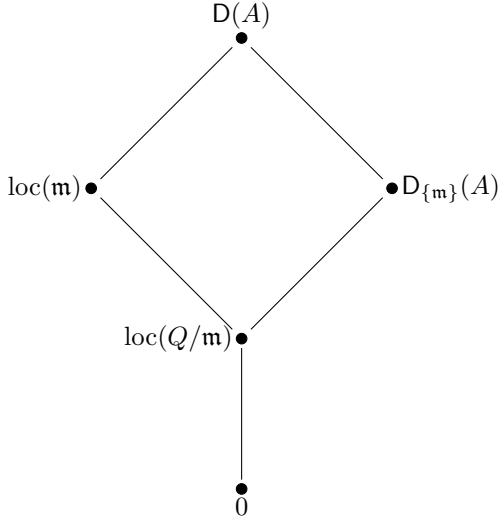
\begin{displaymath}
	\begin{tikzpicture}
    
		\node (v-1) at (0,-2) {};
		\node (v0) at (0,0) {};
    \node (va) at (-2,2) {};
		\node (vc) at (2,2) {};
		\node (vt) at (0,4) {};

    \draw[fill] (v-1)  circle (2pt) node [below] {$0$};
		\draw[fill] (v0)  circle (2pt) node [left] {$\loc(Q/\mfm)$};
		\draw[fill] (va)  circle (2pt) node [left] {$\loc(\mfm)$};
		\draw[fill] (vc)  circle (2pt) node [right] {$\sfD_{\{\mfm\}}(A)$};
		\draw[fill] (vt)  circle (2pt) node [above] {$\sfD(A)$};

    \path[-] (v-1) edge  node [above] {} (v0);
		\path[-] (v0) edge  node [above] {} (va);
		\path[-] (v0) edge  node [above] {} (vc);
		\path[-] (vc) edge  node [above] {} (vt);
		\path[-] (va) edge  node [above] {} (vt);

  \end{tikzpicture}
\end{displaymath}

This frame is finite and so is not only spatial but is a coherent frame. We see that the meet-prime smashing ideals are $\sfP = \loc(\mfm), \sfQ = \sfD_{\{\mfm\}}(A)$, and $0$. The open subsets of $\Spc^\mathrm{s}(\sfD(A))$ are
\[
U_0 = \varnothing, U_{\loc(Q/\mfm)} = \{0\}, U_{\loc(\mfm)} = \{0, \sfQ\}, U_{\sfD_{\{\mfm\}}(A)} = \{0, \sfP\}, \text{ and } U_{\sfD(A)} = \Spc^\mathrm{s}(\sfD(A)).
\]
Thus the smashing spectrum is
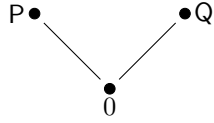
\begin{displaymath}
	\begin{tikzpicture}

		\node (v0) at (0,0) {};
    \node (va) at (-1,1) {};
		\node (vc) at (1,1) {};

		\draw[fill] (v0)  circle (2pt) node [below] {$0$};
		\draw[fill] (va)  circle (2pt) node [left] {$\sfP$};
		\draw[fill] (vc)  circle (2pt) node [right] {$\sfQ$};

		\path[-] (v0) edge  node [above] {} (va);
		\path[-] (v0) edge  node [above] {} (vc);

  \end{tikzpicture}
\end{displaymath}
where $\sfP$ and $\sfQ$ are closed points and $0$ is open (remember that we are in the Hochster dual picture and so to compare with $\Spec (A)$ we should dualize to get a local space with two generic points).

Next let us turn to describing the small support. The points $\sfP$ and $\sfQ$ are closed and so
\[
\Gamma_\sfP = k \text{ and } \Gamma_\sfQ = Q
\]
are the corresponding right idempotents. The open point $0$ is generic, so $V_0 = \Spc^\mathrm{s}(\sfD(A))$ and we can write $\{0\} = U_{\loc(Q/\mfm)}\cap V_0$ yielding
\[
\Gamma_0 = Q/\mfm.
\]
Thus, for instance, if $a\in A$ is a non-zero divisor then the small (and big) support of the perfect complex $A/(a)$ is
\[
\ssupp ( A/(a) ) = \{\sfP, 0\} = U_{\sfD_{\{\mfm\}}(A)}.
\]
We note this is not a closed subset (again one should expect it to be open, as it is, rather than closed since we are in the Hochster dual picture to $\Spc$).

\begin{rem}\label{rem:notsmalltop}
If one topologizes the smashing spectrum only using the supports of compact objects then one does not get the correct space. Indeed, by compatibility with suspension and cones we have for perfect complexes $t,s \in \sfD^\mathrm{perf}(A)$ that
\[
\thick(t) = \thick(s) \text{ implies } \sSupp (t) = \sSupp (s)
\]
(or use Theorem~\ref{thm:comparison}). Thus the possible smashing supports of objects of $\sfD^\mathrm{perf}(A)$ are
\[
\varnothing, \Spc^\mathrm{s}(\sfD(A)), \text{ and } \{\sfP, 0\}
\]
and the corresponding space is not even sober.

\end{rem}

\begin{rem}\label{rem:ssupppreimage}
In this example the small smashing support is not compatible with the comparison map $\psi$. We have $\supp (\mfm) = \Spc (\sfD^\mathrm{perf}(A))$ but
\[
\ssupp (\mfm) = U_{\loc(\mfm)} = \{0, \sfQ\} \subsetneq \Spc^\mathrm{s}(\sfD(A)) = \psi^{-1} (\Spc (\sfD^\mathrm{perf}(A)))^\vee,
\]
where the first equality is Proposition~\ref{prop:ssuppE}.
\end{rem}

\begin{rem}
By assigning some notion of dimension to points of the Balmer spectrum one can, at least when the dual is $T_D$, use the corresponding filtration to decompose the category $\sfT$ into pieces supported at individual primes. This has been done by various authors, with varying levels of sophistication used in reconstructing $\sfT$ from said pieces.

In the cases where the telescope conjecture fails, such as in the current example, one may instead consider decomposing the category $\sfT$ over the smashing spectrum. This gives a finer decomposition.

For instance, let us meditate briefly on the case of $\sfD(A)$, as considered above. Working with the spectrum of the compacts we get the recollement corresponding to the triangle
\[
\Gamma_{\mfm}A \to A \to Q
\]
i.e.\ the gluing of $\Spec A$ from the generic point and the formal scheme at the closed point. On the other hand, when working with the smashing spectrum, this refines to the recollement corresponding to the triangle
\[
\Sigma^{-1}Q/\mfm \to A \to Q\times k.
\]
In some sense we have `gone deeper' than the closed point and the residue field has emerged as generic information.

More generally, when we know $\Spcs(\sfT)$ exists as a space, we can apply much of the machinery that has already been developed for the Balmer spectrum, for instance \cites{BaFilt,StevensonDimension,BalchinGreenlees}, to produce such refined decompositions more generally. In fact, one can already do this at the level of frames, and so being spatial should be inessential.
\end{rem}


\subsection{Big primes}


One gets all known localizing subcategories from $\Spc^\mathrm{s}(\sfD(A))$ by looking at arbitrary subsets, i.e.\ in terms of the small smashing support, but there is currently no classification of localizing ideals. Given this (we don't even know there is a set of localizing subcategories) $\SPC (\sfD(A))$ is a bit more delicate. For instance, what we have in the noetherian setting doesn't generalize naively:
\begin{displaymath}
\loc(k\oplus Q) = \sfD(k)\times \sfD(Q) \subsetneq \sfD(A).
\end{displaymath}
It would be very interesting to compute $\SPC (\sfD(A))$, or at least some of it, in this example. We know at least two primes through the following easy lemma.

\begin{lem}
Suppose $\sfL$ is a localizing ideal in a big tt-category $\sfT$. If $\sfL$ is a maximal ideal, i.e.\ $\sfT/\sfL$ is tt-minimal, then $\sfL$ is prime.
\end{lem}
\begin{proof}
First note that maximality guarantees $\sfL$ is radical. Indeed, if it were not then we would have a non-zero $X\in \sfT/\sfL$ with $X\otimes X = 0$. In particular, $\ker X\otimes(-)$ is non-zero and so must contain $\unit$. Hence $X \cong 0$, which is a contradiction. 

Suppose then that $\sfI,\sfJ$ are radical localizing ideals with $\sfI \cap \sfJ \subseteq \sfL$. If $\sfI \nsubseteq \sfL$ then $\loc(\sfL, \sfI) = \sfT$. Hence
\[
\sfJ = \sfJ \otimes \sfT = \sfJ \otimes \loc(\sfL,\sfI) = \loc(\sfJ \otimes \sfL, \sfJ\otimes \sfI) \subseteq \loc(\sfL)
\]
where the final equality holds as $\sfJ \otimes \sfI \subseteq \sfJ\cap \sfI \subseteq \sfL$.
\end{proof}

Since $\sfD(k)$ and $\sfD(Q)$ are minimal it follows that $\loc(\mfm)$ and $\sfD_{\{\mfm\}}(A)$ are big primes, lying over $0$ and $\sfD^\mathrm{perf}_{\{\mfm\}}(A)$ respectively. One is tempted to suspect that $\loc(Q\oplus k)$ is also maximal and hence prime, but we don't know this.



\appendix

\section{Setting the record straight}\label{sec:fail}

In this appendix we record some results from the previous version, which may still be useful, as well as explicitly pointing out the errors in the `proof' of spatiality in our previous work and in Wagstaffe's thesis \cite{WagstaffeThesis}, along with counterexamples to the methods of proof.

\subsection{Some lemmas which are correct}

Let $\sfT$ be a big tt-category. We record a few preliminary facts about smashing ideals, which were used in the original `proof' in the hope they might still find application in further study of the smashing frame.

\begin{lem}\label{lem:colim}
Let $\{\sfS_i \mid i\in I\}$ be a chain of smashing ideals with corresponding left idempotents $\{E_i \mid i\in I\}$. Then the flat left idempotent in $\Modu \sfT^c$ corresponding to $\sfS = \vee_i \sfS_i$ is $E = \colim_i E_i$. Similarly, if $F_i$ are the corresponding right idempotents for the $\sfS_i$ then $\colim_i F_i$ is the right idempotent for $\sfS$.
\end{lem}
\begin{proof}
Let us start by making sense of the claimed colimit. If $i<j$ in $I$ then $\sfS_i \subseteq \sfS_j$ and so $E_i \otimes E_j \cong E_i$. The structure maps for the colimit are then given by tensoring $\varepsilon_i\colon E_i \to \unit$ with $E_j$ and using the above isomorphism to obtain a map $E_i \to E_j$.

As filtered colimits in $\Modu \sfT^c$ are exact, and the tensor product commutes with colimits, the colimit $E$ is still flat. The universal property of the colimit determines a morphism $\varepsilon\colon E\to \unit$ and one verifies readily that $E$ is a left idempotent. 

Since $E\otimes E_i \cong E_i$ we see that $\im(E\otimes-)$ contains each $\sfS_i$ and hence contains $\sfS$. On the other hand, if $E'$ is the left idempotent corresponding to the join $\sfS$ then $E'\otimes E_i \cong E_i$ and the colimit formula yields $E' \otimes E \cong E$. It follows that $\im(E\otimes-) \subseteq \sfS$ and so $E\cong E'$.

One can prove the statement for right idempotents similarly (see also Remark~\ref{rem:rightyidem}).
\end{proof}

\begin{rem}\label{rem:rightyidem}
Since $I$ is linearly ordered any tensor product $F_{i_1}\otimes \ldots \otimes F_{i_n}$ with $i_r \in I$ is isomorphic to $F_j$ where $j = \max\{i_1,\ldots i_n\}$. By \cite{BKSframe}*{Proposition~5.2} we can write $F = \colim_J F_J$ where $J$ runs over the finite subsets of $I$ and $F_J = \otimes_{j\in J}F_j$. It follows that, by collapsing the indexing set to a final subset to eliminate redundancy, we can write $F = \colim_i F_i$.
\end{rem}

\begin{lem}\label{lem:fp1}
Let $\{\mcM_i \mid i\in I\}$ be a chain of Serre subcategories of $\modu \sfT^c$. Then 
\[
\serre(\cup_i \mcM_i) = \cup_i \mcM_i
\]
i.e.\ the union is already a Serre subcategory. In particular, if
\[
M \in \vee_i \mcM_i = \serre(\cup_i \mcM_i)
\]
then there is a $j\in I$ such that $M\in \mcM_j$.
\end{lem}
\begin{proof}
Let $M' \to M \to M''$ be a short exact sequence in $\modu \sfT^c$. Suppose that $M'$ and $M''$ lie in $\mcM = \cup_i \mcM_i$. We can pick a sufficiently large $j$ such that $M'$ and $M''$ lie in $\mcM_j$. Then, since $\mcM_j$ is Serre, we have $M \in \mcM_j$. As $\mcM_j$ is contained in $\mcM$ we thus see $M\in \mcM$. A similar argument shows that if $M\in \mcM$ then $M'$ and $M''$ are in $\mcM$ and so it is a Serre subcategory as claimed.
\end{proof}

\begin{lem}\label{lem:join}
Let $\{\sfS_i \mid i\in I\}$ be a set of smashing ideals with corresponding right idempotents $\{F_i \mid i\in I\}$ and join $\sfS$ with right idempotent $F$. Then $\mcB_F$ is the smallest localizing Serre subcategory containing $\cup_i \mcB_{F_i}$.
\end{lem}
\begin{proof}
Let us denote by $\mcM$ the localizing Serre subcategory generated by the $\mcB_{F_i}$. It is evident that $\mcM \subseteq \mcB_F$.

Any localizing Serre subcategory of $\Modu \sfT^c$ is determined by the class of injective objects in its right perpendicular category (see \cite{MR0389953}*{Chapter VI.3}). Thus it is sufficient to show these classes for $\mcM$ and $\mcB_F$, which we denote by $I_\mcM$ and $I_{\mcB_F}$ respectively, agree. Observe that
\[
I_\mcM = \cap_i I_{\mcB_{F_i}} \text{ and } I_{\mcB_F} \subseteq I_\mcM
\]
the former by construction and the latter since $\mcM \subseteq \mcB_F$.

Suppose that $J\in I_\mcM$ and note that since $J$ is injective it comes from an object of $\sfT$ which we also call $J$ (this is a consequence of Brown representability, see e.g.\ \cite{KrTele}*{Lemma~1.7}). By our description of $I_\mcM$ the injective $J$ is in the right perpendicular of each $\mcB_{F_i}$.  By \cite{KrCQ}*{Theorem~12.1 (1)} we can identify $\sfS_i^\perp$ with those objects which the restricted Yoneda functor sends to the right perpendicular of $\mcB_{F_i}$ (being in Krause's $\mathfrak{I}^\perp$ is equivalent to being right orthogonal to the localizing Serre subcategory generated by the images of the maps in $\mathfrak{I}$). So we see $J\in \sfS_i^\perp$, and hence  $J\otimes E_i \cong 0$, for each of the left idempotents $E_i$ corresponding to the $\sfS_i$. It follows, using the description of $F$ as a colimit from \cite{BKSframe}*{Proposition~5.2} and flatness of $J$, that $J\otimes F \cong J$. Invoking flatness of $J$ again we deduce that $E\otimes J \cong 0$ where $E$ is the left idempotent for $\sfS$. We claim this implies that $J\in I_{\mcB_F}$. We have shown $E\otimes J \cong 0$ and so $J\in \sfS^\perp$. The claim then follows from another application of \cite{KrCQ}*{Theorem~12.1}. Hence $I_{\mcB_F} = I_\mcM$ as desired.
\end{proof}

\begin{lem}\label{lem:fp2}
Let $\{\sfS_i \mid i\in I\}$ be a chain of smashing ideals with corresponding right idempotents $\{F_i \mid i\in I\}$ and join $\sfS$ with right idempotent $F$. If $M\in \modu \sfT^c$ satisfies $M\otimes F \cong 0$ then there is an $i\in I$ such that $M\otimes F_i \cong 0$.
\end{lem}
\begin{proof}
By Lemma~\ref{lem:join} the localizing Serre ideal $\mcB_F$ of $\Modu \sfT^c$ is generated by $\cup_i \mcB_{F_i}$. The top-left of Reminder~\ref{rem:isos} reminds us that each $\mcB_{F_i}$ is generated by its subcategory of finitely presented objects $\mcB^\fp_{F_i}$. In particular $\mcB_F$ is generated by $\cup_i \mcB^\fp_{F_i}$. We now use that $I$ is a chain: Lemma~\ref{lem:fp1} tells us that this union is already a Serre subcategory and so must be $\mcB^\fp_F$ (which in turn determines $\mcB_F$ by closing under filtered colimits). 

Thus, if $M \in \mcB^\fp_F$ then there is an $i\in I$ such that $M\in \mcB^\fp_{F_i}$, i.e.\ $F_i\otimes M\cong 0$. 
\end{proof}

\subsection{The bit that isn't correct}

The problem with the original argument is the following statement:

\begin{lemnt}\label{notalemma}
Let $\sfS$ be a smashing ideal with right idempotent $F$ and $M\in \modu \sfT^c$ a finitely presented module such that $M\notin \mcB_F$. Then there is a (maximal) meet-prime smashing ideal $\sfP$ with right idempotent $H$ such that $\sfS\subseteq \sfP$ and $M \notin \mcB_H$.
\end{lemnt}

The issue with the proof was an incorrect identification $\mcB_{F_1} \cap \mcB_{F_2} = \mcB_{F_1} \wedge \mcB_{F_2}$ for a pair of right idempotents $F_1$ and $F_2$. This identification holds when restricted to flat objects in these subcategories, but is not true for general objects. Let us give a counterexample to both this identification, and to the above claim.

\begin{ex}\label{counter1}
Consider $\sfD^\mathrm{b}(\Coh \PP^1_k)$ the bounded derived category of coherent sheaves over $\PP^1_k$ for some field $k$ (which we now drop from the notation). Pick a non-trivial map $f\colon \mcO \to \Sigma\mcO(-2)$, and note that $f$ vanishes on any proper open subset of $\PP^1$. In particular, since the telescope conjecture holds it is killed by any non-trivial monoidal smashing localization (where by non-trivial we mean actually inverting some non-invertible map).

We are then led to consider $M = \im(H_f)$ the image of $f$ in $\Modu \sfD^\mathrm{b}(\Coh \PP^1)$. By what we have observed above any flat right idempotent, except for the unit, will kill $M$. Thus, for instance, if $U = \PP^1 \setminus \{0\}$ and $V = \PP^1 \setminus \{\infty\}$ are the standard opens then $M \in \mcB_{\mcO_U} \cap \mcB_{\mcO_V}$. On the other hand, it is not in the meet which is $\mcB_{\mcO_{\PP^1}} = 0$.

Of course it gets worse. Let us try to apply (not a) Lemma~\ref{notalemma} to the smashing ideal $0$, using the above $M$. Any non-trivial smashing localization factors via a localization removing a single point (which of course is not unique and may depend on the localization) from $\PP^1$, and the corresponding Serre ideal will contain $M$, i.e.\ the localization will kill $f$. So the maximal smashing ideal whose associated Serre ideal doesn’t contain M must be 0.

Alas, $0$ is not prime: we have that $\loc(k(0))$ and $\loc(k(\infty))$ are smashing ideals which are non-zero but their intersection is trivial (we already saw this above). Whoops…
\end{ex}

As emphasized in the original version of this article Rose Wagstaffe, in her thesis \cite[Theorem 6.2.2]{WagstaffeThesis}, gives an independent argument that implies the frame of smashing ideals is spatial using model theoretic techniques. Unfortunately, the same issue appears in Wagstaffe's 6.1.22. The issue is the reference to \cite{KrTele}*{Lemma 4.11} which is used to argue that the union of two closed tt-Ziegler subsets is closed. In that article Krause uses the wrong definition of exact ideal, i.e.\ one that doesn't actually correspond to a smashing localization (some discussion of this is in text following \cite{KrCQ}*{Corollary 12.5}). The problem is that an intersection of exact ideals (with the correct definition) need not be exact.

\begin{ex}\label{counter2}
Again, it gets worse. The closed subsets relevant to \cite[Theorem 6.2.2]{WagstaffeThesis} are not closed under binary unions as we now show. This is not intended to be read independently of Wagstaffe's proof and so we don't repeat the necessary set-up from her work. In order to make the comparison clear we will also use some model theoretic language below (without defining it) but the interested reader who is unfamiliar with these terms will find an excellent introduction in Wagstaffe's thesis.

The flavor of the example is similar to Example~\ref{counter1} and we again work with $\sfT = \sfD(\QCoh \PP^1)$ and consider the non-zero map $f\colon \mcO \to \Sigma\mcO(-2)$ corresponding to the Euler sequence, which vanishes when restricting to any proper open subset of $\PP^1$. Let $F$ be the coherent functor given by taking the image of $f$, i.e.\ take the cokernel of the image of the map $\Sigma\mcO(-2) \to \Sigma \mcO(-1)^2$ under the contravariant Yoneda embedding. (Thus $F$ lives in the category of additive functors $[\sfT^c, \Modu \ZZ]$, i.e.\ of left rather than right $\sfT^c$-modules, and is in an appropriate sense dual to the $M$ of Example~\ref{counter1}.)

Consider the definable subcategory $\sfD$ defined by $F$ and all its suspensions. This is determined by the indecomposable pure-injectives $P$ such that 
\[
\xymatrix{
\RHom(\Sigma\mcO(-2), P) \ar[rr]^-{\RHom(f, P)} && \RHom(\mcO, P)
}
\]
is zero (really we only need this map to be zero on cohomology, but this distinction won't come up below).

The indecomposable pure-injective objects of $\sfT$ are completely understood, see for instance \cite{KS17}*{Proposition~4.4.1}, and are given by the indecomposable torsion sheaves, the Pr\"{u}fer sheaves, the adic sheaves, the sheaf of rational functions, and the line bundles $\mcO(i)$. The indecomposable torsion sheaves, the Pr\"{u}fer sheaves, the adic sheaves, and the sheaf of rational functions all come from proper open subsets of $\PP^1$ via pushforward and so they will all satisfy $\RHom(f,P)=0$. On the other hand, using the Euler sequence and the fact that $\RHom(\mcO(-2), \mcO(-1)) = 0$ we see that $\mcO(-2)$ is not in the given definable subcategory, since $\RHom(f, \mcO(-2))$ is an isomorphism.

We have constructed for ourselves a proper definable subcategory $\sfD$ which contains all torsion sheaves, Pr\"{u}fer sheaves, adics sheaves, and the sheaf of rational functions. Now consider the standard affine open cover $\PP^1 = U \cup V$, so $U$ and $V$ are each a copy of $\AA^1$ obtained by excluding $0$ and $\infty$ respectively. The corresponding subcategories $\sfD(\QCoh U)$ and $\sfD(\QCoh V)$ of $\sfT$ are definable, and hence are determined by the pure-injectives they contain. These are precisely the sheaf of rational functions, and the adic, Pr\"{u}fer, and torsion sheaves for points in $U$ and $V$ respectively.

The argument in Wagstaffe's 6.1.22 asserts that the definable hull of $\sfD(\QCoh U) \cup \sfD(\QCoh V)$ is again a definable tensor ideal. Since $\PP^1 = U\cup V$ this would force the definable hull to be $\sfT$ as it contains all the torsion sheaves. On the other hand, both $\sfD(\QCoh U)$ and $\sfD(\QCoh V)$ are contained in the $\sfD$ we constructed above and $\sfD \subsetneq \sfT$. Hence the definable hull of $\sfD(\QCoh U) \cup \sfD(\QCoh V)$ is a proper subcategory of $\sfT$ and so cannot be a tensor ideal.
\end{ex}

%
%
%


\bibliography{greg_bib}

\def\cprime{$'$}
\begin{bibdiv}
\begin{biblist}

\bib{AullThron}{article}{
      author={Aull, C.~E.},
      author={Thron, W.~J.},
       title={Separation axioms between {$T_{0}$} and {$T_{1}$}},
        date={1962},
     journal={Nederl. Akad. Wetensch. Proc. Ser. A 65 = Indag. Math.},
      volume={24},
       pages={26\ndash 37},
}

\bib{BaSpec}{article}{
      author={Balmer, Paul},
       title={The spectrum of prime ideals in tensor triangulated categories},
        date={2005},
     journal={J. Reine Angew. Math.},
      volume={588},
       pages={149\ndash 168},
}

\bib{BaFilt}{article}{
      author={Balmer, Paul},
       title={Supports and filtrations in algebraic geometry and modular
  representation theory},
        date={2007},
        ISSN={0002-9327},
     journal={Amer. J. Math.},
      volume={129},
      number={5},
       pages={1227\ndash 1250},
}

\bib{BaHomological}{article}{
      author={Balmer, Paul},
       title={Homological support of big objects in tensor-triangulated
  categories},
        date={2020},
     journal={J. \'{E}c. polytech. Math.},
      volume={7},
       pages={1069\ndash 1088},
}

\bib{BaNilpotence}{article}{
      author={Balmer, Paul},
       title={Nilpotence theorems via homological residue fields},
        date={2020},
     journal={Tunis. J. Math.},
      volume={2},
      number={2},
       pages={359\ndash 378},
}

\bib{Bazzoni15}{article}{
      author={Bazzoni, S.},
       title={Cotilting modules and homological ring epimorphisms},
        date={2015},
        ISSN={0021-8693},
     journal={J. Algebra},
      volume={441},
       pages={552\ndash 581},
         url={https://doi.org/10.1016/j.jalgebra.2015.07.008},
      review={\MR{3391937}},
}

\bib{BaRickard}{article}{
      author={Balmer, Paul},
      author={Favi, Giordano},
       title={Generalized tensor idempotents and the telescope conjecture},
        date={2011},
     journal={Proc. Lond. Math. Soc. (3)},
      volume={102},
      number={6},
       pages={1161\ndash 1185},
}

\bib{BalchinGreenlees}{article}{
      author={Balchin, Scott},
      author={Greenlees, J. P.~C.},
       title={Adelic models of tensor-triangulated categories},
        date={2020},
        ISSN={0001-8708},
     journal={Adv. Math.},
      volume={375},
       pages={107339, 45},
         url={https://doi.org/10.1016/j.aim.2020.107339},
}

\bib{BHS2021}{article}{
      author={Barthel, Tobias},
      author={Heard, Drew},
      author={Sanders, Beren},
       title={Stratification and the comparison between homological and tensor
  triangular support},
        date={2021},
     journal={arXiv preprint arXiv:2106.16011},
}

\bib{BarthelStrat}{article}{
      author={Barthel, Tobias},
      author={Heard, Drew},
      author={Sanders, Beren},
       title={Stratification in tensor triangular geometry with applications to
  spectral {M}ackey functors},
        date={2021},
     journal={arXiv preprint arXiv:2106.15540},
}

\bib{BKSfield}{article}{
      author={Balmer, Paul},
      author={Krause, Henning},
      author={Stevenson, Greg},
       title={Tensor-triangular fields: ruminations},
        date={2019},
     journal={Selecta Mathematica},
      volume={25},
      number={1},
       pages={1\ndash 36},
}

\bib{BKSframe}{article}{
      author={Balmer, Paul},
      author={Krause, Henning},
      author={Stevenson, Greg},
       title={The frame of smashing tensor-ideals},
        date={2020},
        ISSN={0305-0041},
     journal={Math. Proc. Cambridge Philos. Soc.},
      volume={168},
      number={2},
       pages={323\ndash 343},
         url={https://doi.org/10.1017/s0305004118000725},
}

\bib{BazzoniStovicek}{article}{
      author={Bazzoni, Silvana},
      author={\v{S}t'ov\'{\i}\v{c}ek, Jan},
       title={Smashing localizations of rings of weak global dimension at most
  one},
        date={2017},
     journal={Adv. Math.},
      volume={305},
       pages={351\ndash 401},
}

\bib{Complex}{article}{
      author={Clausen, Dustin},
      author={Scholze, Peter},
       title={Condensed mathematics and complex geometry},
        date={2022},
     journal={Notes available at
  https://people.mpim-bonn.mpg.de/scholze/Complex.pdf},
}

\bib{DP08}{article}{
      author={Dwyer, W.~G.},
      author={Palmieri, J.~H.},
       title={The {B}ousfield lattice for truncated polynomial algebras},
        date={2008},
     journal={Homology Homotopy Appl.},
      volume={10},
      number={1},
       pages={413\ndash 436},
}

\bib{SpectralBook}{book}{
      author={Dickmann, Max},
      author={Schwartz, Niels},
      author={Tressl, Marcus},
       title={Spectral spaces},
      series={New Mathematical Monographs},
   publisher={Cambridge University Press, Cambridge},
        date={2019},
      volume={35},
        ISBN={978-1-107-14672-3},
         url={https://doi.org/10.1017/9781316543870},
}

\bib{Foxby}{article}{
      author={Foxby, Hans-Bj{\o}rn},
       title={Bounded complexes of flat modules},
        date={1979},
        ISSN={0022-4049},
     journal={J. Pure Appl. Algebra},
      volume={15},
      number={2},
       pages={149\ndash 172},
         url={https://doi.org/10.1016/0022-4049(79)90030-6},
}

\bib{HochsterSpectral}{article}{
      author={Hochster, M.},
       title={Prime ideal structure in commutative rings},
        date={1969},
        ISSN={0002-9947},
     journal={Trans. Amer. Math. Soc.},
      volume={142},
       pages={43\ndash 60},
}

\bib{Johnstone}{book}{
      author={Johnstone, Peter~T.},
       title={Stone spaces},
      series={Cambridge Studies in Advanced Mathematics},
   publisher={Cambridge University Press, Cambridge},
        date={1982},
      volume={3},
        ISBN={0-521-23893-5},
}

\bib{KockPitsch}{article}{
      author={Kock, Joachim},
      author={Pitsch, Wolfgang},
       title={Hochster duality in derived categories and point-free
  reconstruction of schemes},
        date={2017},
        ISSN={0002-9947},
     journal={Trans. Amer. Math. Soc.},
      volume={369},
      number={1},
       pages={223\ndash 261},
         url={https://doi.org/10.1090/tran/6773},
}

\bib{KrTele}{article}{
      author={Krause, Henning},
       title={Smashing subcategories and the telescope conjecture---an
  algebraic approach},
        date={2000},
        ISSN={0020-9910},
     journal={Invent. Math.},
      volume={139},
      number={1},
       pages={99\ndash 133},
}

\bib{KrCQ}{article}{
      author={Krause, Henning},
       title={Cohomological quotients and smashing localizations},
        date={2005},
        ISSN={0002-9327},
     journal={Amer. J. Math.},
      volume={127},
      number={6},
       pages={1191\ndash 1246},
}

\bib{KrAuslander}{article}{
      author={Krause, Henning},
       title={Deriving {A}uslander's formula},
        date={2015},
     journal={Doc. Math.},
      volume={20},
       pages={669\ndash 688},
}

\bib{KS17}{incollection}{
      author={Krause, H.},
      author={Stevenson, G.},
       title={The derived category of the projective line},
        date={2019},
   booktitle={Spectral structures and topological methods in mathematics},
      series={EMS Ser. Congr. Rep.},
   publisher={EMS Publ. House, Z\"{u}rich},
       pages={275\ndash 297},
}

\bib{NeeChro}{article}{
      author={Neeman, Amnon},
       title={The chromatic tower for {$D(R)$}},
        date={1992},
     journal={Topology},
      volume={31},
      number={3},
       pages={519\ndash 532},
        note={With an appendix by Marcel B{\"o}kstedt},
}

\bib{MR2868166}{book}{
      author={Picado, Jorge},
      author={Pultr, Ale\v{s}},
       title={Frames and locales},
      series={Frontiers in Mathematics},
   publisher={Birkh\"{a}user/Springer Basel AG, Basel},
        date={2012},
        ISBN={978-3-0348-0153-9},
         url={https://doi.org/10.1007/978-3-0348-0154-6},
        note={Topology without points},
}

\bib{StevensonActions}{article}{
      author={Stevenson, Greg},
       title={Support theory via actions of tensor triangulated categories},
        date={2013},
     journal={J. Reine Angew. Math.},
      volume={681},
       pages={219\ndash 254},
}

\bib{StevensonAbsFlat}{article}{
      author={Stevenson, Greg},
       title={Derived categories of absolutely flat rings},
        date={2014},
     journal={Homology Homotopy Appl.},
      volume={16},
      number={2},
       pages={45\ndash 64},
}

\bib{StevensonDimension}{article}{
      author={Stevenson, Greg},
       title={The local-to-global principle for triangulated categories via
  dimension functions},
        date={2017},
        ISSN={0021-8693},
     journal={J. Algebra},
      volume={473},
       pages={406\ndash 429},
         url={https://doi.org/10.1016/j.jalgebra.2016.12.002},
}

\bib{MR0389953}{book}{
      author={Stenstr\"{o}m, Bo},
       title={Rings of quotients},
      series={Die Grundlehren der mathematischen Wissenschaften, Band 217},
   publisher={Springer-Verlag, New York-Heidelberg},
        date={1975},
        note={An introduction to methods of ring theory},
}

\bib{WagstaffeThesis}{thesis}{
      author={Wagstaffe, Rose},
       title={Definability in monoidal additive and tensor triangulated
  categories},
        type={Ph.D. Thesis},
        date={2021},
}

\end{biblist}
\end{bibdiv}

\end{document}